\documentclass[11pt]{amsart} 
\usepackage[utf8]{inputenc}
\usepackage{csquotes}
\usepackage[english]{babel}
\usepackage[hmargin=2.3cm,vmargin=3.8cm]{geometry} 
\usepackage[utf8]{inputenc}

\usepackage{tikz}
\usepackage{tikz-cd}
\usepackage{amsmath,dsfont,amsthm,amsfonts,enumerate,xcolor,stmaryrd,mathrsfs,amssymb,graphicx,mathtools}
\usepackage{thmtools}
\usepackage[colorlinks,citecolor=blue,linkcolor=blue,urlcolor=blue,filecolor=blue,breaklinks]{hyperref}

\usepackage{epsfig}
\usepackage{epstopdf}
\usepackage{subcaption}
\usepackage{fancyhdr}

\definecolor{dOrange}{rgb}{0.9,0.4,0}

\fancypagestyle{firstpage}{\rfoot{\thepage}\chead{}\lhead{}}

\usetikzlibrary{matrix}
\usepackage{thm-restate}
 \usepackage[backend=biber, style=alphabetic, sorting=nyt, doi=false, isbn=false, eprint=false, related=false]{biblatex}
\definecolor{dOrange}{rgb}{0.9,0.4,0}

\DeclareFieldFormat{url}{Available at\addcolon\space\url{#1}}

\AtEveryBibitem{%
  \ifentrytype{misc}{%
  }{%
    \clearfield{url}%
    \clearfield{urldate}%
  }%
}

\renewbibmacro{in:}{%
  \ifboolexpr{%
     test {\ifentrytype{article}}%
     or
     test {\ifentrytype{inproceedings}}%
  }{}{\printtext{\bibstring{in}\intitlepunct}}%
}

\makeatletter
\newcommand{\pushright}[1]{\ifmeasuring@#1\else\omit\hfill$\displaystyle#1$\fi\ignorespaces}
\newcommand{\pushleft}[1]{\ifmeasuring@#1\else\omit$\displaystyle#1$\hfill\fi\ignorespaces}
\makeatother

\DeclareMathOperator{\Aut}{Aut}
\DeclareMathOperator{\hAut}{hAut}
\DeclareMathOperator{\Out}{Out}
\DeclareMathOperator{\Inn}{Inn}
\DeclareMathOperator{\Hom}{Hom}
\DeclareMathOperator{\Diff}{Diff}

\DeclareMathOperator{\cGL}{GL}
\DeclareMathOperator{\cG}{\mathcal{G}}

\DeclareMathOperator{\A}{\mathcal{A}}
\DeclareMathOperator{\Ch}{Ch}
\DeclareMathOperator{\coCh}{coCh}

\DeclareMathOperator{\id}{id}
\DeclareMathOperator{\Ab}{Ab}
\DeclareMathOperator{\sgn}{sgn}
\DeclareMathOperator{\Ind}{Ind}
\DeclareMathOperator{\Mod}{Mod}
\DeclareMathOperator{\Fun}{Fun}

\newcommand*{\fullref}[1]{\hyperref[{#1}]{\ref*{#1}. \nameref*{#1}}} 

\newcommand{\bQ}{\mathbb{Q}}
\newcommand{\bZ}{\mathbb{Z}}
\newcommand{\cC}{\mathcal{C}}
\newcommand{\cE}{\mathcal{E}}
\newcommand{\stable}{\mathrm{st}}
\newcommand{\Tw}{\mathrm{Tw}}
\newcommand{\Gr}{\mathrm{Gr}}

\newcommand{\coker}{\mathrm{coker}}
\newcommand{\cP}{\mathcal{P}}
\newcommand{\Inj}{\mathrm{Inj}}
\newcommand{\Image}{\mathrm{im}}

\theoremstyle{plain}
  \newtheorem{thm}{Theorem}[section]
  \newtheorem{theorem}[thm]{Theorem}
  \newtheorem{corollary}[thm]{Corollary}
  \newtheorem{proposition}[thm]{Proposition}
  \newtheorem{lemma}[thm]{Lemma}

 \newtheorem{thmxx}{Theorem}

\theoremstyle{definition}

  \newtheorem{definition}[thm]{Definition}
  
    \newtheorem{remark}[thm]{Remark}
    \newtheorem{notation}[thm]{Notation}
    \newtheorem{assumption}[thm]{Assumption}

  \graphicspath{{Images/}{\main/Images/}}

\newcommand{\Addresses}{{
  \bigskip
  \footnotesize
  
\textsc{Erik Lindell, Institut for Matematiske Fag, Københavns Universitet, Universitetsparken 5, 2100
København Ø, Denmark.}\par\nopagebreak
  \textit{E-mail address:} \texttt{erikjlindell@gmail.com}
}}

\title{\textbf{Stable cohomology of} $\mathrm{{Aut}(F_n)}$\textbf{ with bivariant twisted coefficients}}
\author{Erik Lindell }
\date{}

\begin{document}

\maketitle

\begin{abstract}
    \noindent We compute the cohomology groups of the automorphism group of the free group $F_n$, with coefficients in arbitrary tensor products of the standard rational representation $H_1(F_n,\bQ)$ and its dual, in a range where $n$ is sufficiently large compared to the cohomological degree and the number of tensor factors.
\end{abstract}

\thispagestyle{firstpage}

\phantomsection\label{sec:introduction} 

\section{Introduction}\label{sec:introduction-name}

\noindent Let $F_n$ denote the free group on $n\ge 1$ generators and let $\Aut(F_n)$ denote its automorphism group. There is a group homomorphism $\Aut(F_n)\to\Aut(F_{n+1})$ given by extending an automorphism of $F_n$ to $F_{n+1}$ so that it acts trivially on the new generator. Several authors have considered homological stability for this sequence, with various coefficients.  Homological stability with $\bQ$-coefficients was proven by Hatcher and Vogtmann \cite{HatcherVogtmannCerfTheory} and later with $\bZ$-coefficients by the same authors \cite{HatcherVogtmannOut}. The stable rational cohomology was computed and proven to be trivial by Galatius \cite{Galatius11}, confirming a conjecture by Hatcher and Vogtmann \cite{HatcherVogtmannConjecture}. In this paper we compute the stable cohomology with rational \textit{bivariant twisted coefficients}. To define these, we start by letting $H_\bZ(n):=H_1(F_n;\bZ)$. For $p,q\ge 0$, we define 
$$K^{p,q}_\bZ(n):=H_\bZ(n)^{\otimes p}\otimes H^*_\bZ(n)^{\otimes q},$$
where we write $H_\bZ^*(n):=\Hom_\bZ(H_\bZ(n),\bZ)$ for the linear dual.

\begin{remark}
    The terminology ``bivariant'' coefficients comes from the language of functor homology, since $H_\bZ(n)$ defines the components of a covariant functor from the category of finitely generated free groups to the category of abelian groups, whereas $H^*_\bZ(n)$ defines the components of a contravariant functor on this category.
\end{remark}

\noindent Letting $\Sigma_k$ denote the symmetric group on $k$ letters, note that $\Sigma_p\times\Sigma_q$ acts on $K^{p,q}_\bZ(n)$ by permuting the tensor factors. Since this action commutes with the diagonal $\Aut(F_n)$-action, this induces an action on the cohomology groups as well.

The inclusion $F_n\hookrightarrow F_{n+1}$ induces a $\bZ[\Aut(F_n)]$-linear map $H^*_\bZ(n+1)\to H^*_\bZ(n)$, where the source is considered a $\bZ[\Aut(F_n)]$-module via the stabilization map $\Aut(F_n)\to\Aut(F_{n+1})$. Since $F_{n+1}\cong F_n*\bZ$, the inclusion $F_n\hookrightarrow F_{n+1}$ has a retraction given by the projection $F_{n}*\bZ\to F_n$. This induces an $\bZ[\Aut(F_n)]$-linear map $H_\bZ(n+1)\to H_\bZ(n)$. Applying these factor-wise gives us a $\bZ[\Aut(F_n)]$-linear map $K^{p,q}_\bZ(n+1)\to K_\bZ^{p,q}(n)$ for every $p,q\ge 0$, which induces maps
\begin{equation}\label{eq:Kpq-stabilization}
    H^*\left(\Aut(F_{n+1});K_\bZ^{p,q}(n+1)\right)\to H^*\left(\Aut(F_n);K_\bZ^{p,q}(n)\right)
\end{equation}
on cohomology groups. 

The coefficients $K^{p,q}_\bZ(n)$ are examples of so called \textit{coefficients of finite degree}, so it follows from results of Randal-Williams and Wahl \cite[Theorem G]{WahlRW} that the map \eqref{eq:Kpq-stabilization} is an isomorphism for $n$ sufficiently large in comparison to $*$, $p$ and $q$. We define the \textit{stable} cohomology as the limit 
$$H_{\stable}^*\left(\Aut(F_\infty),K_\bZ^{p,q}\right):=\lim_{n\to\infty }H^*(\Aut(F_n),K_\bZ^{p,q}(n)),$$
where the limit is taken with respect to the maps \eqref{eq:Kpq-stabilization}.

In this paper we will work over $\bQ$, so we let $H(n):=H_\bZ(n)\otimes\bQ$ and $K^{p,q}(n):=K_\bZ^{p,q}(n)\otimes\bQ$. The goal of the paper is to compute $H_\stable^*(\Aut(F_\infty),K^{p,q})$ (defined similarly as the integral stable cohomology), for any $p,q\ge 0$.

There are several previous results in this direction. Using functor homology methods, it was proven by Djament and Vespa \cite[Théorème 1]{DjamentVespa} that for all $q>0$, $H^*_\stable(\Aut(F_\infty),H^*_\bZ(\infty)^{\otimes q})=0$ (note that this result is over the integers). This was also shown independently by Randal-Williams \cite[Theorem A(i)]{RW18}, using geometric methods. 

With the coefficients $H(n)^{\otimes p}$, the stable cohomology was computed by Randal-Williams \cite[Theorem A(ii)]{RW18} and independently by Vespa \cite[Theorem 4]{Vespa18}, by applying a result of Djáment \cite[Théorème 1.10]{DjamentHodge}.

The stable cohomology with the coefficients $K^{p,q}(n)$, for general $p,q\ge 0$, has also been studied by Djament, who gave a conjectural description of the stable cohomology groups \cite[Théorème 7.4]{DjamentHodge}. The main result of this paper confirms the conjecture of Djament, essentially by pushing the methods of Randal-Williams a bit further.

To describe the stable cohomology groups, let us denote by $\cP_{p,q}$ the vector space spanned by partitions of the set $\{1,2,\ldots, p\}$, with at least $q$ parts, $q$ of which are labeled from 1 to $q$ and the remaining parts unlabeled. The symmetric groups $\Sigma_p$ and $\Sigma_q$ act on $\cP_{p,q}$ by permuting the elements of $\{1,\ldots,p\}$ and the labels. Letting $\sgn_p$ and $\sgn_q$ denote the sign representations of $\Sigma_p$ and $\Sigma_q$ respectively, we may now state our main theorem as follows:

\begin{thmxx} For all $p,q\ge 0$ and $2{*}\le n-p-q-3$ we have 
\begin{align*}
    H^*\left(\Aut(F_n);K^{p,q}(n)\right)\cong\begin{cases}\cP_{p,q}\otimes\sgn_p\otimes\sgn_q&\text{if }*=p-q\\ 0&\text{ otherwise},\end{cases}
\end{align*}
as $\Sigma_p\times\Sigma_q$-representations.\label{thm:A}
\end{thmxx}

\noindent The stable range in Theorem \ref{thm:A} follows directly from \cite[Theorem 3.2]{RW18}, combined with \cite[Definition 3.3]{RW18} and the discussion after Remark 3.4 of the same paper, which shows that $K_\bZ^{p,q}(n)$ defines a (split) coefficient system of degree $p+q$. For this reason, we will not go into the proof of this range in more detail in this paper.

\subsection{Some relations to other results} Let us briefly describe some connections of Theorem \ref{thm:A} to other results in the litterature.

\subsubsection{Relation to Djament's conjecture} The conjecture of Djament \cite[Théorème 7.4]{DjamentHodge} is stated in quite different terms than Theorem \ref{thm:A}. Letting $\mathrm{gr}$ denote the category of finitely generated free groups and group homomorphisms and $\mathcal{F}(\mathrm{gr})$ the category of functors $\mathrm{gr}\to\Ab$, where $\Ab$ is the category of abelian groups, let $\mathfrak{a}_\bQ$ denote the \textit{rational abelianization functor}, i.e.\ abelianization postcomposed with tensoring with $\bQ$, and for $p\ge 0$, let $T^p,\Lambda^p:\Ab\to\Ab$ denote the $p$th tensor and exterior power functors, respectively. Djament's conjecture is stated in terms of homology, but using \cite[Equation 4]{Vespa18}, we obtain that the dual statement of the conjecture is that
\begin{equation}\label{eq:Djament's-conj}
    H^*_{\mathrm{st}}(\Aut(F_\infty),K^{p,q})\cong\bigoplus_{i+j=*} \mathrm{Ext}^i_{\mathcal{F}(\mathrm{gr})}\left((T^q\circ\mathfrak{a}_\bQ)\otimes(\Lambda^j\circ\mathfrak{a}_\bQ),T^{p}\circ\mathfrak{a}_\bQ\right).
\end{equation}
Using \cite[Theorem 1]{Vespa18}, we have that as vector spaces
\begin{align*}
    \mathrm{Ext}^i_{\mathcal{F}(\mathrm{gr})}\left((T^q\circ\mathfrak{a}_\bQ)\otimes(\Lambda^j\circ\mathfrak{a}_\bQ),T^{p}\circ\mathfrak{a}_\bQ\right)\cong\begin{cases}
        \bQ[\mathrm{Surj}(p,q+j)/\Sigma_j]&\text{ if }i+j=p-q\\
        0&\text{ otherwise},
    \end{cases}
\end{align*}
where $\mathrm{Surj}(p,q+j)$ is the set of surjections from $\{1,\ldots,p\}$ to $\{1,\ldots,q+j\}$ and where $\Sigma_j$ acts by postcomposition by bijections via the standard injection $\Sigma_j\hookrightarrow\Sigma_{q+j}$. As representations, this isomorphism holds up to sign representations of $\Sigma_p$ and $\Sigma_q$, described by \cite[Proposition 2.5]{Vespa18}. There is an isomorphism of $\Sigma_{p}\times\Sigma_q$-representation from the vector space on $\mathrm{Surj}(p,q+j)/\Sigma_j$, with the obvious $\Sigma_p\times\Sigma_q$-action, to $\cP(p,q)$ given by sending the equivalence class of a surjection $f$ to the labeled partition 
$$\{(f^{-1}(1),1),\ldots,(f^{-1}(q),q), f^{-1}(q+1),\ldots, f^{-1}(q+j)\},$$
giving us the connection between the description \eqref{eq:Djament's-conj} and that of Theorem \ref{thm:A}.

\subsubsection{The wheeled PROP structure} The stable cohomology groups of Theorem \ref{thm:A} were also recently studied by Kawazumi and Vespa in
\cite{KawazumiVespa}, who proved that the collection of the stable cohomology groups, for all $p, q \ge 0$, can be
given the structure of a \textit{wheeled PROP} in the category of graded $\bQ$-vector spaces (for a definition, see \cite{MarklMerkulovShadrin}). In these terms, the conjecture by Djament can be stated by
saying that this is the wheeled PROP generated by one graded commutative generator $h_1$ in bi-arity
$(p, q) = (2, 1)$ and degree 1, satisfying the relation 
$$(h_1\otimes 1)\circ h_1+(1\otimes h_1)\circ h_1=0,$$
where $1$ is the identity operation, $\otimes$ is the horizontal composition and $\circ$ is the vertical composition (this relation may be called ``graded associativity''). The conjectural generator was introduced by Kawazumi in \cite{Kawazumi05} and independently
by Farb and Cohen-Pakianathan (unpublished). In their paper, Kawazumi and Vespa define an injective morphism of wheeled PROP from the wheeled PROP generated by one operation as above, satisfying the stated relations, to the wheeled PROP formed by the stable cohomology groups. Combining this with Theorem \ref{thm:A}, it follows that this is indeed an isomorphism of wheeled PROPs.

\subsubsection{Cohomology of the $\mathrm{IA}$-automorphism group} The conjecture of Djament also appears in recent work by Habiro and Katada \cite[Conjecture 9.4]{HabiroKatada}, where the authors prove that the combination of this conjecture with two other conjectures produces a conjectural description of the stable rational cohomology of the subgroup $\mathrm{IA}_n\subseteq \Aut(F_n)$, which is the kernel of the homomorphism
$$\Aut(F_n)\to\Aut(F_n^{\mathrm{ab}})\cong\cGL_n(\bZ).$$
This approach was also pushed a bit further in recent work of the author \cite{Lindell24Walled}.

\subsection{Cohomology of outer automorphism groups} Theorem \ref{thm:A} can also be used to obtain some results about the cohomology of $\Out(F_n):=\Aut(F_n)/\Inn(F_n)$, where $\Inn(F_n)$ denotes the subgroup of inner automorphisms. Since inner automorphisms of a group act trivially on the (co)homology of that group, the $\Aut(F_n)$-representations $H(n)$ and $H^*(n)$ factor through $\Out(F_n)$, so we may consider the cohomology groups
$$H^*(\Out(F_n);K^{p,q}(n))$$
for any $p,q\ge 0$. It follows from a theorem of Kawazumi \cite[Theorem 7.1]{Kawazumi05} that for any $\Out(F_n)$-representation $M$, we have
\begin{equation}\label{eq:Kawazumi-decomposition}
    H^*(\Aut(F_n);M)\cong H^*(\Out(F_n);M)\oplus H^{*-1}(\Out(F_n);M\otimes H^*(n)).
\end{equation}
In particular, the cohomology of $\Out(F_n)$ with coefficients in $M$ is a summand of that of $\Aut(F_n)$ with the same coefficients, and the pullback map in cohomology induced by the quotient homomorphism $\Aut(F_n)\to\Out(F_n)$ is injective. This gives us the following immediate corollary of Theorem \ref{thm:A}:

\begin{corollary}
    If $*\neq p-q$ and $2{*}\le n-p-q-3$ we have
    $$H^*(\Out(F_n);K^{p,q}(n))=0,$$
    for all $p,q\ge 0$.
\end{corollary}

\noindent In the cases where either $p=0$ or $q=0$, this result was proven by Randal-Williams \cite[Theorem B]{RW18}. He also proved the following:

\begin{theorem}{\cite[Theorem B(ii)]{RW18}}\label{thm:RW-Out}
    If $n\ge 4p+3$, we have that $H^p(\Out(F_n);K^{p,0}(n))$ is isomorphic to the subrepresentation of $\cP_{p,0}\otimes\sgn_p$ spanned by partitions with no parts of size 1.
\end{theorem}

\noindent Using Equation \eqref{eq:Kawazumi-decomposition}, we get that for any $p,q\ge 0$
    $$\dim H^{*}(\Out(F_n);K^{p,q}(n))= \dim H^{*+1}(\Aut(F_n);K^{p,q-1}(n))-\dim H^{*+1}(\Out(F_n);K^{p,q-1}(n)),$$
    in a stable range $n\gg *+p+q$. Using Theorem \ref{thm:RW-Out} and Theorem \ref{thm:A}, we can thus recursively compute the dimension of $H^{p-q}(\Out(F_n),K^{p,q}(n))$. However, it seems more difficult to obtain a closed formula for the dimension. It seems even more difficult to determine the stable structure of $H^{p-q}(\Out(F_n),K^{p,q}(n))$ as a representation of $\Sigma_p\times\Sigma_q$ and more difficult still to determine the wheeled PROP-structure on the collection of all of the stable cohomology groups.

\subsection{Overview of the paper}\label{sec:overview} In Section \ref{sec:prel}, we first recall some preliminaries about induced representations of different symmetric groups. We then introduce a number of permutation modules of labeled partitions, which is what is used to compute the explicit description of the stable cohomology groups in Theorem \ref{thm:A}. The most important part of this subsection is Proposition \ref{prop:filtered-spectral-seq}, which describes a spectral sequence in group cohomology associated to a short exact sequence of representations. This section can be skipped on a first reading and referred back to when necessary.

In Section \ref{sec:markedgraphs}, we recall two different topological characterizations of the groups 
$$A_n^s:=\Aut(F_n)\ltimes F_n^{\times (s-1)},$$
defined for $s\ge 1$, in terms of mapping class groups of graphs and 3-manifolds. We also introduce for each $s\ge 1$ an $\bZ[A_n^s]$-module $H_\bZ(n,s):=H_1(F_{n+s-1};\bZ)$, on which $A_n^s$ acts via a certain homomorphism $A_n^s\to\Aut(F_{n+s-1})$ defined below (see \eqref{eq:alphans} below).

The steps of the proof of Theorem \ref{thm:A} can be summarized as follows: \begin{enumerate}[I.]\itemsep0.5em 
    \item We start by considering the inclusion
    $$\mu_n^s:A_n^s\to A_n^{s+1}.$$
    The projection $F_{n+s}\cong F_{n+s-1}*\bZ\to F_{n+s-1}$ induces an $A_n^s$-equivariant map $H_\bZ(n,s+1)\to H_\bZ(n,s)$ and the first step of the proof is to show that the induced map
    $$H^*(A_n^{s+1},H_\bZ(n,s+1)^{\otimes p})\to H^*(A_n^s, H_\bZ(n,s)^{\otimes p})$$
    is an isomorphism in a stable range $n\gg *+p$. For $s=1$, we have $A_n^s=\Aut(F_n)$ and $H_\bZ(n,s)=H_\bZ(n)$, and as the stable cohomology groups with the coefficients in tensor powers of $H(n)$ are already known by the results of Djament-Vespa and Randal-Williams cited above, the stability result of this step allows us to compute $H^*(A_n^s,H(n,s)^{\otimes p})$, in a stable range $n\gg *+p$, for any $s\ge 1$.
    
    This is an unsurprising result for experts and is proven using standard methods due to Randal-Williams and Wahl \cite{WahlRW}, combined with an idea by Hatcher and Wahl \cite{HatcherWahlErratum}. Since the proof is still quite long and requires a bit more background, it is left to the final Section \ref{sec:stability}.
    \item In the second step, we consider the inclusion $H(n,1)\to H(n,s+1)$ induced by the homomorphism $F_n\hookrightarrow F_{n+s}$. This map is in fact $A_n^{s+1}$-equivariant, with the action on the source induced by the projection $A_n^{s+1}\to\Aut(F_n)$. The action of $A_n^{s+1}$ on $H(n,s+1)$ is defined in such a way that it acts trivially on the cokernel of the inclusion map, so we get a short exact sequence 
    $$0\to H(n,1)\to H(n,s+1)\to\bQ^{s}\to 0$$
    of $A_n^{s+1}$-representations. Applying Lemma \ref{lemma:extension} to this short exact sequence gives us a spectral sequence, which allows us to use the result of the first step to compute $H^*(A_n^{s+1},H(n,1)^{\otimes p})$, for any $p\ge 0$, in a stable range $n\gg *+p$.
    \item In the third step, we consider the short exact sequence
    $$1\to F_{n}^{\times s}\to A_n^{s+1}\to \Aut(F_n)\to 1$$
    and its associated Hochshild-Serre spectral sequence with coefficients twisted by the $A_n^{s+1}$-representation $H(n,1)^{\otimes p}$. As this representation factors through $\Aut(F_n)$, the spectral sequence has the form
    $$E_2^{i,j}=H^i(\Aut(F_n), H^j(F_n^{\times s},\bQ)\otimes H(n,1)^{\otimes p})\Rightarrow H^{i+j}\left(A_n^{s+1},H(n,1)^{\otimes p}\right).$$
    We prove that this spectral sequence degenerates on the $E_2$-page, so the result of the previous step allows us to compute the direct sum of the terms of its $E_2$-page, in a stable range.
    \item In the final step of the proof, we use the fact that $H^j(F_n^{\times s},\bQ)$ decomposes as a direct sum of copies of $H^*(n)^{\otimes j}$ and again that the stable cohomology groups of $\Aut(F_n)$ with coefficients in $H(n)^{\otimes p}$ are known from the results of Djament-Vespa and Randal-Williams, to inductively compute $H^i(\Aut(F_n),K^{p,j}(n))$, in a stable range, from the result of the previous step.
\end{enumerate}

\noindent The first step of the proof is explained in more detail at the end of Section \ref{sec:markedgraphs}, whereas the steps II-IV are explained in Section \ref{sec:proofofthmA}.

\subsection{Acknowledgements} I would like to thank my PhD supervisor Dan Petersen, whose help was invaluable throughout this project. I would also like to thank Oscar Randal-Williams, Arthur Soulié, Christine Vespa and Nathalie Wahl for their interest in the project and for a number of very useful discussions. In addition, I would like to thank Fabian Hebestreit, Manuel Krannich and Andrew Putman for useful feedback and discussion in conjunction with the defense of my PhD thesis, where an earlier version of this paper was included and where they were on the committee, together with Christine Vespa. I am very grateful for the large number of very helpful comments on earlier versions of this paper that I have received from all of the above, as well as from Mai Katada. I would finally like to thank the referee for carefully reading the paper and coming with many useful comments.

\section{Preliminaries}\label{sec:prel}

\subsection{Conventions} Let us fix some conventions: \begin{itemize}
    \item We fix $\bQ$ as our ground field. In particular, all (co)homology is taken with rational coefficients, unless stated otherwise.
    \item For $S$ a set, we write $\bQ S$ for the vector space generated by $S$.
\item For a positive integer $k$, we write $[k]:=\{1,\ldots,k\}$.

\end{itemize}

\subsection{Induced representations of symmetric groups} \label{sec:inducedreps} For non-negative integers $i,j$, consider the standard embedding $\Sigma_i\times\Sigma_j\hookrightarrow\Sigma_{i+j}$. At several points below we consider representations of symmetric groups $\Sigma_{k}$ that arise by induction of representations of $\Sigma_i\times\Sigma_j$, with $i+j=k$. In this section we recall the definition and describe the structure of these representations in some additional detail.  

Given a $\Sigma_i\times\Sigma_j$-representation $V$, we define 
$$\Ind_{\Sigma_i\times\Sigma_j}^{\Sigma_{i+j}} V:=\bQ[\Sigma_{i+j}]\otimes_{\bQ[\Sigma_i\times\Sigma_j]} V.$$
A full set of representatives of the cosets in $\Sigma_{i+j}/(\Sigma_i\times\Sigma_j)$ is given by choosing, for each partition of the set $[i+j]$ into a disjoint union $I\sqcup J$ with $|I|=i$ and $|J|=j$, a permutation $\sigma_{I,J}$ such that $\sigma_{I,J}(\{1,\ldots,i\})=I$ and $\sigma_{I,J}(\{i+1,\ldots,i+j\})=J$. In other words, for every $\sigma\in\Sigma_{i+j}$ there is some $\sigma_{I,J}$ and $h\in\Sigma_{i}\times\Sigma_j$ such that $\sigma=\sigma_{I,J}h$.

This means that as a vector space we have
$$\Ind_{\Sigma_i\times\Sigma_j}^{\Sigma_{i+j}}V\cong\bigoplus_{\substack{I\sqcup J=[i+j]\\|I|=i}}\bQ\{\sigma_{I,J}\}\otimes V$$
and with the $\Sigma_{i+j}$-action described as follows: if $v\in V$ and $\sigma\in\Sigma_{i,j}$, there is some $\sigma_{I',J'}$ and $h\in \Sigma_i\times\Sigma_j$ such that $\sigma\cdot\sigma_{I,J}=\sigma_{I',J'}\cdot h$. We thus have 
\begin{equation}\label{eq:inductionaction}
    \sigma\cdot(\sigma_{I,J}\otimes v)=(\sigma_{I',J'}\cdot h)\otimes v=\sigma_{I',J'}\otimes (hv).
\end{equation}

\subsection{Labeled partitions} Next, we introduce some permutation modules of labeled partitions that we use below to obtain the explicit description of the cohomology groups in Theorem \ref{thm:A}. 

\begin{definition}
    For finite sets $S,T$, we define the following $\Sigma_S\times\Sigma_T$-representations:\begin{enumerate}
        \item $\cP_{S,\le T}$ is the $\bQ$-vector space with basis consisting of partitions of $S$ with some of the parts labeled by elements of $T$ and the remaining parts unlabeled and where we require that no two parts are labeled by the same element of $T$.
        \item $\cP_{S,T}$ is the subspace of $\cP_{S,\le T}$ spanned by those labeled partitions of $S$ with precisely $|T|$ labeled parts.
        \item $\overline{\cP}_{S,\le T}$ is the subspace of $\cP_{S,\le T}$ spanned by the labeled partitions of $S$ with no unlabeled parts.
    \end{enumerate}
    In each case, there is an obvious $\Sigma_S\times\Sigma_T$-action given by permutation. If either $S=[k]$ or $T=[l]$ for some $k,l\ge 0$, we replace $S$ and $T$ in the notation by $k$ or $l$, respectively.
\end{definition}

\begin{lemma}\label{lemma:partitions-tensor-prod}
    Suppose that $S,T$ finite sets, we have $\overline{\cP}_{S,\le T}\cong(\bQ^{\times T})^{\otimes S}$, as $\Sigma_S\times\Sigma_T$-representations.
\end{lemma}

\begin{proof}
    Let us write $\{e_t\mid t\in T\}$ for the canonical basis elements of $\bQ^{\times T}$ and suppose that $S=\{s_1,\ldots,s_k\}$. We define a map $(\bQ^{\times T})^{\otimes S}\to\overline\cP_{S,T}$ which on the basis elements is given by sending a pure tensor $e_{t_{s_1}}\otimes\cdots\otimes e_{t_{s_k}}$, to the partition of $S$ where we define $s,s'$ to be in the same part if in the corresponding tensor factors, $t_{s}=t_{s'}$, and we decorate this part by $t_{s}$. This map is clearly equivariant and it has an inverse defined on the basis of labeled partitions by sending a labeled partition to the tensor which is $e_{t}$ in the tensor factors corresponding to $s\in S$, if $s$ is in the part of the labeled partition which is labeled by $t$.
\end{proof}

\begin{lemma}\label{lemma:partitions-ind}
    We have isomorphisms of $\Sigma_p\times\Sigma_q$-representations\begin{enumerate}[(i)]
        \item\label{partitions-ind1} $\cP_{p,\le q}\cong\bigoplus_{k=0}^q\Ind_{\Sigma_k\times\Sigma_{q-k}}^{\Sigma_q}\cP_{p,k}$,
        \item\label{partitions-ind2}  $\cP_{p,\le q}\cong\bigoplus_{k=0}^p\Ind_{\Sigma_k\times\Sigma_{p-k}}^{\Sigma_p}(\cP_{k,0}\otimes\overline{\cP}_{p-k,\le q})$.
    \end{enumerate}
\end{lemma}

\begin{proof}
For \eqref{partitions-ind1}, we start by introducing the notation that for $k\le q$, $\cP_{p,k\le q}$ is the permutation module on the set of labeled partitions of $[p]$ with at least $k$ parts, labeled by $k$ distinct elements of $[q]$, and the remaining parts unlabeled. It is thus immediate that as $\Sigma_p\times\Sigma_q$-representations
\begin{equation}\label{eq:decomp-of-P_p,le-q}
    \cP_{p,\le q}\cong \bigoplus_{k=0}^q \cP_{p,k\le q}.
\end{equation}
As vector spaces we have
\begin{equation}\label{eq:P_p,k-le-q-as-vector-space}
    \cP_{p,k\le q}\cong \bigoplus_{\substack{I\subseteq [q]\\|I|=k}}\cP_{p,I},
\end{equation}
with the obvious $\Sigma_p\times\Sigma_q$-action given by permutation. By the description of Section \ref{sec:inducedreps}, we have that for any $k\le q$
\begin{equation}\label{eq:induction-of-P_p,k}
    \Ind_{\Sigma_k\times\Sigma_{q-k}}^{\Sigma_q}\cP_{p,k}\cong\bigoplus_{\substack{I\sqcup J= [q]\\|I|=k}}\bQ\{\sigma_{I,J}\}\otimes \cP_{p,k},
\end{equation}
with the $\Sigma_q$-action described by Equation \eqref{eq:inductionaction}. For any $I\subseteq [q]$ with $|I|=k$, we have a bijective linear map $\bQ\{\sigma_{I,J}\}\otimes\cP_{p,k}\to\cP_{p,I}$ given by applying $\sigma_{I,J}$ to the labels of a labeled partition, and these assemble to a map
$$\Ind_{\Sigma_{k}\times\Sigma_{q-k}}^{\Sigma_q}\cP_{p,k}\to \cP_{p,k\le q},$$
which is also bijective by combining the isomorphisms \eqref{eq:P_p,k-le-q-as-vector-space} and \eqref{eq:induction-of-P_p,k}. It is also easily verified to be $\Sigma_p\times\Sigma_q$-equivariant, using Equation \eqref{eq:inductionaction}. Thus \eqref{partitions-ind1} follows by the isomorphism \eqref{eq:decomp-of-P_p,le-q}.

For \eqref{partitions-ind2}, we use similarly that
\begin{align*}
    \Ind_{\Sigma_{k}\times\Sigma_{p-k}}^{\Sigma_p}(\cP_{k,0}\otimes\overline{\cP}_{p-k,\le q})\cong \bigoplus_{\substack{I\sqcup J=[p]\\|I|=k}}\bQ\{\sigma_{I,J}\}
\end{align*}
so applying $\sigma_{I,J}$ to the partitions of $\cP_{k,0}$ and $\overline{\cP}_{p-k,\le q}$, we get
\begin{align*}
     \bigoplus_{k=0}^p\Ind_{\Sigma_{k}\times\Sigma_{p-k}}^{\Sigma_p}(\cP_{k,0}\otimes\overline{\cP}_{p-k,\le q})\cong\bigoplus_{k=0}^p\bigoplus_{\substack{I\sqcup J=[p]\\ |I|=k}} (\cP_{I,0}\otimes\overline{\cP}_{J,\le q})\cong \cP_{p,\le q}.\ \ \ \ \qedhere
\end{align*}
\end{proof}

\subsection{A key spectral sequence} In this subsection, we prove the following proposition, which is of key importance in Section \ref{sec:proofofthmA} below:

\begin{proposition}\label{prop:filtered-spectral-seq}
    Let $G$ be a group and 
    $$0\to A\to B\to C\to 0$$
    a short exact sequence of $\bQ[G]$-modules. For any $p\ge 0$, there exists a convergent spectral sequence of the form
    $$E_1^{k,l}=H^{k+l}\left(G,\Ind_{\Sigma_k\times\Sigma_{p-k}}^{\Sigma_p}\left(B^{\otimes k}\otimes (C[1])^{\otimes (p-k)}\right)\right)\Rightarrow H^{k+l}(G,A^{\otimes p}),$$
     where $C[1]$ denotes $C$ put in degree 1 and $\Ind_{\Sigma_k\times\Sigma_{p-k}}^{\Sigma_p}\left(B^{\otimes k}\otimes (C[1])^{\otimes (p-k)}\right)$ is considered as a trivial cochain complex of $\bQ[G]$-modules, concentrated in degree $p-k$.
\end{proposition}

\noindent We prove this using the following lemma:

\begin{lemma}\label{lemma:extension}
    Let $R$ be an associative $\bQ$-algebra and let $\A$ be the category of left $R$-modules. If 
    $$0\to A\to B\to C\to 0$$
    is a short exact sequence in $\A$ and $p\ge 1$, there exists a $\Sigma_p$-equivariant resolution $\cE$ of $A^{\otimes p}$ in the category $\coCh^{\ge }(\A)$ of non-negatively graded cochain complexes in $\A$ and a descending filtration
    $$0=F_{p+1}\cE\subset F_k\cE\subset\cdots\subset F_0\cE=\cE$$
    of $\cE$, where each step is equipped with a $\Sigma_p$-action compatible with that on $\cE$, whose graded pieces are
    $$\Gr_k^F\cE=F_k\cE/F_{k+1}\cE\cong \Ind_{\Sigma_k\times\Sigma_{p-k}}^{\Sigma_p} \left(B^{\otimes k}\otimes (C[1])^{\otimes (p-k)}\right).$$
\end{lemma}

\begin{proof}
    Given a short exact sequence as in the lemma, we first note that we can define a filtration
    $$0=F_{p+1}B^{\otimes p}\subset F_pB^{\otimes p}\subset\cdots\subset F_0B^{\otimes p}=B^{\otimes p}$$
    by
    $$F_kB^{\otimes p}:=\sum_{\substack{I\sqcup J= [p]\\|I|\ge k}}A^{\otimes I}\otimes B^{\otimes J}.$$
    and whose graded pieces thus are
    $$\Gr_k^FB^{\otimes p}=\bigoplus_{\substack{I\sqcup J=[p]\\|I|=k}}A^{\otimes I}\otimes C^{\otimes J}.$$
    Note that there is a $\Sigma_p$-action on each summand of $F_kB^{\otimes p}$, compatible with action on $B^{\otimes p}$ given by permuting the factors. Furthermore, using Equation \eqref{eq:inductionaction} above, we see that we have a $\Sigma_p$-equivariant isomorphism
    $$\Gr_k^FB^{\otimes p}\cong \Ind_{\Sigma_k\times\Sigma_{p-k}}^{\Sigma_p}\left(A^{\otimes k}\otimes C^{\otimes (p-k)}\right).$$
    Next, we note that our short exact sequence in $\A$ defines a quasi-isomorphism of complexes 
    $$A\to(B\to C)$$
    in $\coCh(\A)$, where we consider $A$ as a complex concentrated in degree 0 and $(B\to C)$ as a complex concentrated in degrees 0 and 1. Thus $\cE:=(B\to C)$ is a resolution of $A$. Furthermore, the two-term complex $\cE$ itself sits in a short exact sequence
    $$0\to B\to \cE\to C[1]\to 0$$
    in the category $\coCh(\A)$, where $B$ is considered as a complex concentrated in degree 0. We can thus apply the filtration above to this short exact sequence, giving us precisely a filtration as in the statement of the lemma.
\end{proof}

\begin{proof}[Proof of Proposition \ref{prop:filtered-spectral-seq}]
    Recall (see for example \cite[Chapter VII.5]{BrownGroupCohomology}) that for $\cC\in\coCh^{\ge}(\bQ[G])$ and $\cP$ a projective resolution of $\bQ$ in $\Ch(\bQ[G])$, the category of \textit{chain complexes} of $G$-representations, we have
    $$H^*(G,\cC^\bullet)=H^*(\mathscr{H}om_G(\cP,\cC)),$$
    where $\mathscr{H}om_G(\cP,\cC)$ is the total cochain complex of the double complex with $\Hom_G(\cP_q,\cC^p)$ in bidegree $(p,q)$, and differentials induced by those of $\cC$ and $\cP$. 

    Let $\cE$ and $F_\bullet\cE$ be as in the previous lemma and let $\cP\in\Ch(\bQ[G])$ be a projective resolution of $\bQ$. It follows by projectivity that 
    $$\mathscr{H}om_G(\cP,-):\coCh^{\ge 0}(\bQ[G])\to \coCh^{\ge 0}(\bQ[G])$$
    is an exact functor. Thus we get an induced filtration $\tilde{F}_\bullet\left(\mathscr{H}om_G(\cP,\cE)\right)$  of $\mathscr{H}om_G(\cP,\cE)$ in $\coCh^{\ge 0}(\bQ[G])$ with
    $$\tilde{F}_k\left(\mathscr{H}om_G(\cP,\cE)\right)=\mathscr{H}om_G(\cP,F_k\cE)$$
    and whose graded pieces are
    $$\Gr_k^{\tilde{F}}\left(\mathscr{H}om_G(\cP,\cE)\right)=\mathscr{H}om_G(\cP,\Ind_{\Sigma_k\times\Sigma_{p-k}}^{\Sigma_p}(B^{\otimes k}\otimes C[1]^{\otimes (p-k)})).$$
    Since
    $$H^*(G,A^{\otimes p})\cong H^*(G,\cE)=H^*\left(\mathscr{H}om_G(\cP,\cE)\right)$$
    and 
    $$H^*\left(G,\Ind_{\Sigma_{k}\times\Sigma_{p-k}}^{\Sigma_p}\left(B^{\otimes k}\otimes C[1]^{\otimes (p-k})\right)\right)=H^*\left(\mathscr{H}om_G(\cP,\Ind_{\Sigma_k\times\Sigma_{p-k}}^{\Sigma_p}(B^{\otimes k}\otimes C[1]^{\otimes (p-k)}))\right)$$
    we get the desired spectral sequence by \cite[Theorem 2.6]{McCleary-UsersGuide}.
\end{proof}

\section{Mapping class groups of graphs and 3-manifolds}\label{sec:markedgraphs}

\noindent For $s\ge 1$, let $A_n^s:=\Aut(F_n)\ltimes F_n^{\times (s-1)}$. In this section we recall two different topological characterizations of this group, as well as two different stabilization maps between them. We then state a stability theorem for the cohomology of these groups with certain twisted coefficients, whose proof is the topic of Section \ref{sec:stability}.

\subsection{Mapping class groups of graphs} Our first alternative characterization of $A_n^s$ is in terms of graphs. For $s\ge 1$, we let $R_{n,s}$ denote the space in Figure \ref{fig:ThornedRose}, i.e.\ the wedge sum of $n$ circles with $s-1$ unit intervals, based at one of their end-points. Denoting by $x_1$ the wedge point and $x_2,\ldots,x_s$ the non-basepoint end-points of the intervals, we define
$$G_n^s:=\pi_0\hAut_{\{x_1,\ldots,x_s\}}\left(R_{n,s}\right),$$
where $\hAut_{\{x_1,\ldots,x_s\}}\left(R_{n,s}\right)$ is the topological monoid of homotopy self-equivalences (i.e.\ \textit{homotopy automorphisms}) of $R_{n,s}$ that fix $x_1,\ldots,x_s$ pointwise. We have an isomorphism $G_n^s\cong A_n^s$ (see for example \cite{HatcherWahlBoundaries}).

\begin{figure}[h]
    \centering
    \includegraphics[scale=0.3]{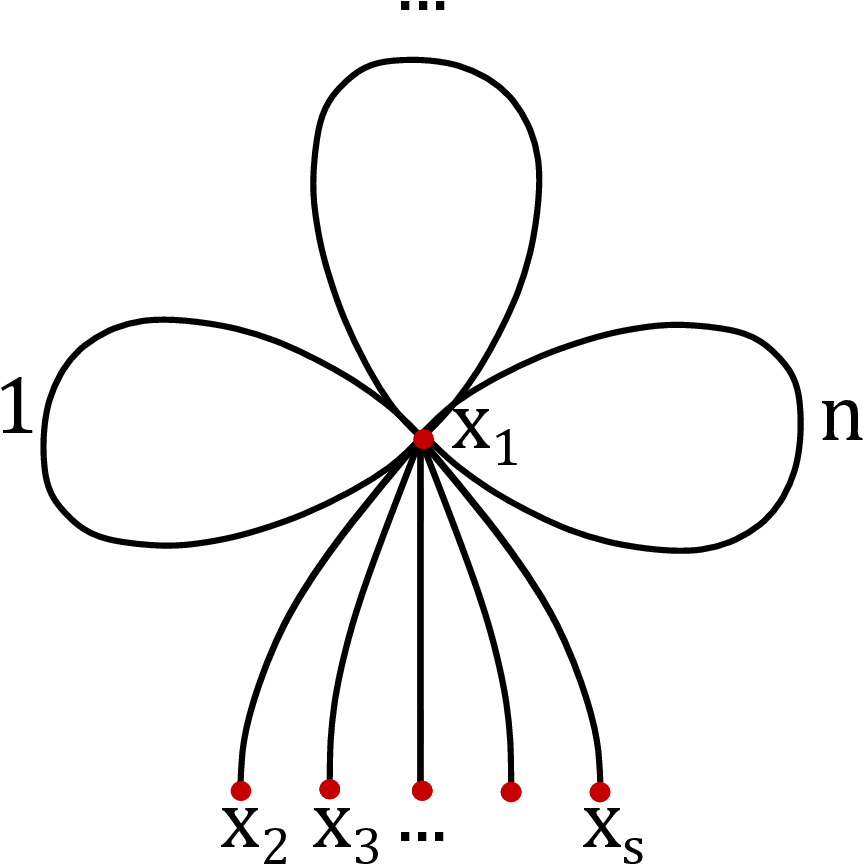}
    \caption{The space $R_{n,s}$.}
    \label{fig:ThornedRose}
\end{figure}

Let us describe the isomorphism $G_n^s\cong A_n^s$ explicitly in terms of generators of $A_n^s$. The first three kinds of generators are of the form $(\phi,1,\ldots,1)$, where $\phi$ is a \textit{Nielsen transformation} in $\Aut(F_n)$:\begin{enumerate}[(i)]
    \item\label{AnsGen1} $\phi\in\Aut(F_n)$ is given by permuting the generators $x_1,\ldots,x_n$ of $F_n$. Under the isomorphism $A_n^s\cong G_n^s$, these correspond to the (classes of) self-homeomorphisms of $R_n^s$ given by permuting the $S^1$-wedge summands in the corresponding way. 
    \item\label{AnsGen2} $\phi\in\Aut(F_n)$ is given by mapping one generator $x_i$ to its inverse $x_i^{-1}$. Under the isomorphism $A_n^s\cong G_n^s$ this corresponds to the (class of the) homeotopy automorphisms which on the $i$th $S^1$-wedge summand  of $R_n^s$ is given by the degree $-1$ pointed self-homeomorphism, and the identity elsewhere.
    \item\label{AnsGen3} $\phi\in\Aut(F_n)$ is given by mapping $x_i$ to $x_jx_i$, for $i\neq j$ and fixing the remaining generators of $F_n$. Under the isomorphism $A_n^s\cong G_n^s$, this corresponds to the (class of the) homotopy automorphism of $R_n^s$ which ``wraps'' the first half of $i$th $S^1$-wedge summand around the $j$th wedge summand, and the second half along the $i$th $S^1$-wedge summand, as illustrated in Figure \ref{Fig:generatorsAns}, and which is the identity elsewhere.
    \item\label{AnsGen4} The remaining generators of $A_n^s$ are given by elements of the form $(\id_{F_n},1,\ldots, x_i,\ldots,1)$, non-trivial in the $k$th $F_n$-component, where $x_i\in F_n$ is one of the generators. Under the isomorphism $A_n^s\cong G_n^s$, the generators of $A_n^s$ which is $x_i$ in the $k$th $F_n$-factor, and the identity element in the remaining factors, corresponds to the (class of the) homotopy automorphism of $R_n^s$ which ``wraps'' the first half of the $k$th $I$-wedge summand around the $i$th $S^1$-wedge summand and sends the second half to the same $I$-wedge summand, as illustrated in Figure \ref{Fig:generatorsAns}.
\end{enumerate}

\begin{figure}[h]
    \centering
    \includegraphics[scale=0.25]{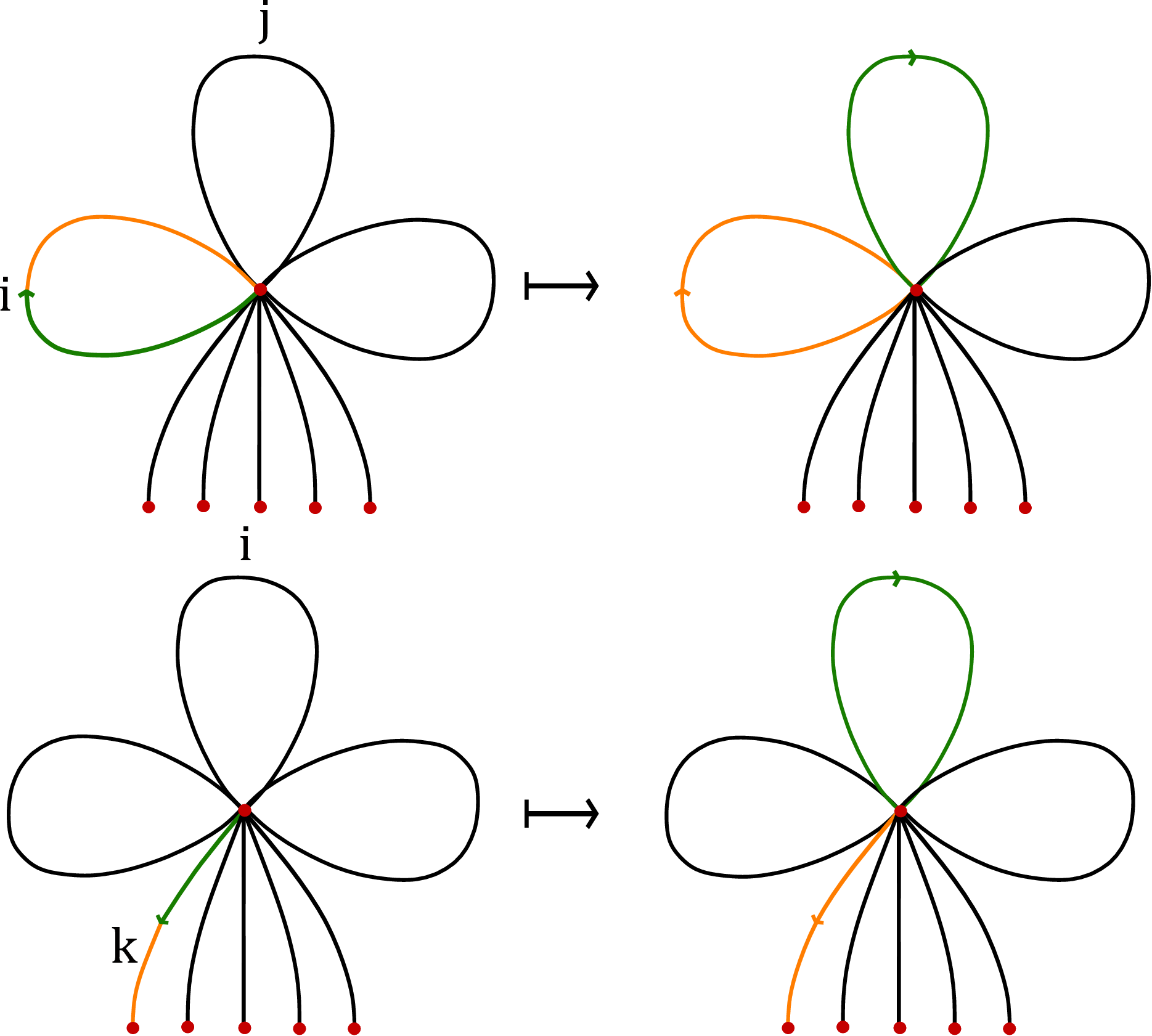}
    \caption{Illustration of the generators of $A_n^s$ of the forms \eqref{AnsGen3} and \eqref{AnsGen4}.}
    \label{Fig:generatorsAns}
\end{figure}

\subsection{Mapping class groups of 3-manifolds}\label{sec:MCG-3mflds} The group $A_n^s$ can also be characterized as a quotient of the mapping class group of the $3$-manifold $M_n^s:=\left(\#^n S^2\times S^1\right)\setminus\left(\sqcup^s\mathring D^3\right)$. Let 
$$\Mod(M_n^s):=\pi_0\Diff_\partial\left(M_n^s\right)$$
be the mapping class group of boundary preserving diffeomorphisms. This group has a normal subgroup generated by \textit{Dehn twists around embedded 2-spheres}, which we will denote by $\Tw(M_n^s)$ (see \cite[Section 2]{HatcherWahl3manifolds} for more details). Let us write $\Gamma_n^s:=\Mod(M_n^s)/\Tw(M_n^s)$. We have
$$\Gamma_n^s\cong G_n^s$$
and thus $A_n^s\cong\Gamma_n^s$ (see \cite[Proposition 1]{HatcherVogtmannOut} for an explicit description of the isomorphism).

\subsection{(Co)homological stability with twisted coefficients} The first step of the proof of Theorem \ref{thm:A} consists of computing the stable cohomology of $A_n^s$ with certain twisted coefficients, generalizing $H(n)$, that we introduce below.

This calculation is done using a (co)homological stability argument similar to the one used by Hatcher, Vogtmann and Wahl in \cite{HatcherVogtmannWahlErratum} to  prove the following:

\begin{theorem}{(\cite[Corollary on page 6]{HatcherVogtmannWahlErratum})}\label{thm:Hatcher-Vogtmann-Wahl}
For $s\ge 1$, let $\mu_n^s:A_n^s\to A_n^{s+1}$ be the inclusion map, induced by the inclusion $F_n^{\times(s-1)}\hookrightarrow F_n^{\times s}$ of the first $s-1$ coordinates. This map induces an isomorphism on integral cohomology in degrees $2{*}\le n-2$.
\end{theorem}

\noindent The idea of their proof is to consider two different stabilization maps that increase $n$, but keep $s$ constant, and which both factor through $\mu_n^s$. Let us introduce these. 

\begin{definition}
    For $s\ge 1$, we define
$$\phi_n^s:A_n^s\to A_{n+1}^s,$$
which on the $\Aut(F_n)$-component is given by extending an automorphism to the new generator by identity, and on the remaining components by the standard inclusion $F_n\hookrightarrow F_{n+1}$.
\end{definition}

\begin{remark}
    In terms of the groups $G_n^s$, $\phi_n^s$ is induced by extending a homotopy equivalence from $R_{n,s}$ to $R_{n+1,s}$ by letting it be the identity on the extra loop. 
\end{remark}

\begin{remark}
    In terms of the groups $\Gamma_n^s$, this stabilization can be defined as follows: we construct $M_{n+1}^s$ by taking the boundary connected sum of $M_n^s$ with the manifold $(S^2\times S^1)\setminus \mathring D^3$ and extending a diffeomorphism along to be the identity on this submanifold. This descends to a homomorphism $\Gamma_n^s\to\Gamma_{n+1}^s$, as it sends Dehn twists along embedded 2-spheres in $M_n^s$ to Dehn twists along embedded 2-spheres in $M_{n+1}^s$.
\end{remark}

\noindent Before defining our second stabilization map, we note that $\phi_n^s$ can be factored as follows: first, for $s\ge 1$, we note that in terms of the groups $G_n^s$, $\mu_n^s:G_n^s\to G_n^{s+1}$ is the map induced by extending a homotopy automorphism from $R_{n,s}$ to $R_{n,s+1}$ by identity. Secondly, for $s\ge 2$ we let 
\begin{equation}\label{eq:alphans}
    \alpha_n^s:G_n^s\to G_{n+1}^{s-1}
\end{equation}
be the map induced by identifying $x_s$ with $x_1$ in $R_{n,s}$ and thereby creating an extra loop. We can now see that $\phi_n^s=\alpha_n^{s+1}\circ\mu_n^s$.  Using our descriptions \eqref{AnsGen3} and \eqref{AnsGen4} of the generators of $A_n^s$, we get the following explicit description of $\alpha_n^s$:

\begin{proposition}
    For  $(\psi,y_1,\ldots,y_s)\in A_n^s$, we have
    $$\alpha_n^s(\psi,y_1,\ldots,y_s)=(L_{y_s,x_{n+1}}\circ\phi_n^1(\psi),y_1,\ldots,y_{s-1}),$$
where $L_{y_s,x_{n+1}}\in\Aut(F_{n+1})$ is defined by $x_{n+1}\mapsto y_sx_{n+1}$ and acting trivially on the remaining generators.
\end{proposition}

\noindent The second stabilization map is now defined simply by composition, but in the opposite order to $\phi_n^s$:

\begin{definition}
    We define $\nu_n^s:=\mu_{n+1}^{s-1}\circ\alpha_n^s:A_n^s\to A_{n+1}^s$.
\end{definition}

\begin{remark}
    In terms of $3$-manifolds, the stabilization map $\nu_n^s$ can be defined as follows: consider a cylinder $S^2\times I$ with a marked 2-disk in each boundary component and glue these to two marked disks in the first and last boundary components of $M_n^s$. We can then once again extend a diffeomorphism by the identity and in the same way as for the first stabilization, this descends to a homomorphism $\Gamma_n^s\to\Gamma_{n+1}^s$.
\end{remark}

\noindent Hatcher, Vogtmann and Wahl proved that for $n$ sufficiently large, both $\phi_n^s$ and $\nu_n^s$ induce isomorphisms on (co)homology with $\bZ$-coefficients (see \cite[Theorem 4]{HatcherVogtmannOut} and \cite[Theorem]{HatcherVogtmannWahlErratum}, respectively). Since $\phi_n^s=\alpha_n^{s+1}\circ\mu_n^s$, it follows that for $n$ sufficiently large in comparison to ${*}$,  $(\mu_n^s)^*:H^*(A_n^{s+1},\bZ)\to H^*(A_n^s,\bZ)$ is surjective (and that $(\alpha_n^{s+1})^*$ is injective) . The corresponding statement for $\nu_n^s$ shows that $(\mu_n^s)^*$ is injective (and that $(\alpha_n^s)$ is surjective). Theorem \ref{thm:Hatcher-Vogtmann-Wahl} thus follows as a corollary.

The stable cohomology of $A_n^1=\Aut(F_n)$ with $\bQ$-coefficients, with respect to the stabilization $\phi_n^1$, was computed by Galatius in \cite{Galatius11} and combining this result with Theorem \ref{thm:Hatcher-Vogtmann-Wahl} allowed him to compute the stable cohomology of $A_n^s$ with $\bQ$-coefficients, with respect to $\phi_n^s$, for any $s\ge 1$. 

In our case, the stable cohomology of $\Aut(F_n)$ with the coefficients $H(n)^{\otimes p}$, for any $p\ge 0$, are known from the results of Djament-Vespa and Randal-Williams, so we use the argument of Hatcher, Vogtmann and Wahl, but with twisted coefficients.

\subsubsection{Twisted coefficients}\label{subsubsec:twistedstabilizationmaps} For $(X,Y)$ a pair of topological spaces, let $H_*(X,Y;\bZ)$ denote the relative homology of $(X,Y)$. We define our twisted coefficients as follows:

\begin{definition}
    For $s\ge 1$, let $H_\bZ(n,s):=H_1(R_{n,s},\{x_1,\ldots,x_s\};\bZ)$ and $H(n,s):=H_\bZ(n,s)\otimes\bQ$.
\end{definition} 

\noindent Note that we have 
    $$H_1(R_{n,s},\{x_1,\ldots,x_s\};\bZ)\cong H_1(R_{n,s}/\{x_1,\ldots,x_s\};\bZ)\cong H_1(R_{n+s-1,1};\bZ).$$
    In terms of group homology, we thus have that $H_\bZ(n,s)\cong H_1(F_{n+s-1};\bZ)$ as an $\bZ[A_n^s]$-module, where for $s\ge 2$, the $A_n^s$-action on the right hand side is defined by the composition 
    $$\alpha_{n+s-2}^2\circ\cdots\circ\alpha_n^s:A_n^s\to\Aut(F_{n+s-1}).$$
We can also define maps between these representations corresponding to the stabilization homomorphisms above. We have a map $R_{n,s+1}/\{x_1,\ldots,x_{s+1}\}\to R_{n,s}/\{x_1,\ldots,x_s\}$, given by mapping the loop coming from the last leg to $x_1$. It induces an $A_n^s$-equivariant map
$$m_n^{s+1}:H_\bZ(n,s+1)\to H_\bZ(n,s),$$
where the $A_n^s$-action on the source is defined via the homomorphism $\mu_n^s:A_n^s\to A_n^{s+1}$. 

Similarly, for $s\ge 2$, there is a map $R_{n+1,s-1}/\{x_1,\ldots,x_{s-1}\}\to R_{n,s}/\{x_1,\ldots,x_s\}$ given by mapping the $n$th loop to the loop corresponding to the $s$th leg. This induces an $A_n^s$-equivariant map
$$a_{n+1}^{s-1}:H_\bZ(n+1,s-1)\to H_\bZ(n,s),$$
where the $A_n^s$-action on the source is defined via the homomorphism $\alpha_n^s:A_n^s\to A_{n+1}^{s-1}$.

We will prove the following theorem:

\begin{theorem}\label{thm:twostabilizations}
    For each $p\ge 0$, there is some $N(p)\in\bZ$ such that whenever $n\ge 2{*}+2p+3$ and $n\ge N(p)$, the maps
    \begin{align*}
        \left(\phi_n^s, (m_n^{s+1}\circ a_{n+1}^s)^{\otimes p}\right)^*:  H^*\left(A_{n+1}^s;H_\bZ(n+1,s)^{\otimes p}\right)&\to H^*\left(A_n^s;H_\bZ(n,s)^{\otimes p}\right)\end{align*}
        and
        \begin{align*}
        \left(\nu_n^s, (a_{n+1}^{s-1}\circ m_{n+1}^{s})^{\otimes p}\right)^*:H^*\left(A_{n+1}^s;H_\bZ(n+1,s)^{\otimes p}\right)&\to H^*\left(A_n^s;H_\bZ(n,s)^{\otimes p}\right)
    \end{align*}
    are both isomorphisms.
\end{theorem}

\begin{remark}
The stable range in Theorem \ref{thm:twostabilizations} produces a stable range for our main result which is worse than the one stated in Theorem \ref{thm:A}. Since we know that there is a better range, as explained above, the specific range in this theorem is only included for completeness and its consequences can be safely ignored in the proofs going forward.

The value of $N(p)$ in this theorem is due to a technical detail (see Corollary \ref{cor:tensorpowercoefficients} below) and is likely not optimal. It has an explicit formula in terms of $p$, but as it is not really relevant going forward we will not consider it further. 
\end{remark}

\noindent The coefficient systems in Theorem \ref{thm:twostabilizations} are examples of so-called \textit{finite degree} coefficient systems. We will prove the theorem using standard methods based on the framework for homological stability which was introduced by Randal-Williams and Wahl \cite{WahlRW} and extended by Krannich \cite{KrannichStability}. The setup and proof requires a few pages however, so we leave this to the final Section \ref{sec:stability}. For the rest of this section and Section \ref{sec:proofofthmA}, we will thus take Theorem \ref{thm:twostabilizations} as given.

By the same line of reasoning as above, we get the following as an immediate corollary:

\begin{corollary}
    For each $p\ge 0$, there is some $N(p)\in\bZ$ such that for each $s\ge 1$, whenever $n\ge 2{*}+2p+3$ and $n\ge N(p)$, the map
    $$(\mu_n^s,m_n^{s+1})^*:H^*\left(A_n^{s+1};H(n,s+1)^{\otimes p}\right)\to H^*\left(A_n^s; H(n,s)^{\otimes p}\right)$$
    is an isomorphism.
\end{corollary}

\noindent In the case $s=1$, the following result was obtained by Djament-Vespa and Randal-Williams:

\begin{theorem}[{\cite[Theorem A(ii)]{RW18}}]\label{PropRW}
For $n\ge 2{*}+p+3$ we have an isomorphism
\begin{align*}
    H^*\left(\Aut(F_n); H(n)^{\otimes p}\right)\cong\begin{cases}\cP_{p,0}\otimes\sgn_p&\text{ if }*=p,\\
    0&\text{ otherwise},\end{cases}
\end{align*}
of $\Sigma_p$-representations.
\end{theorem}

\noindent Combining Theorem \ref{thm:twostabilizations} with Theorem \ref{PropRW}, we get the following corollary:

\begin{corollary}\label{cor:stable-cohomology-1}
For each $p\ge 0$, there is some $N(p)\in\bZ$ such that for each $s\ge 1$, whenever $n\ge 2{*}+2p+3$ and $n\ge N(p)$, we have an isomorphism of $\Sigma_p$-representations
\begin{align*}
    H^*\left(A_n^s; H(n,s)^{\otimes p}\right)\cong\begin{cases}\cP_{p,0}\otimes\sgn_p&\text{ if }*=p,\\
    0&\text{ otherwise}.\end{cases}
\end{align*}
\end{corollary}

\noindent This corollary is the input for the proof of Theorem \ref{thm:A}, which we prove in the next section.

\section{Stable cohomology with bivariant twisted coefficients}\label{sec:proofofthmA}

\noindent In this section, we assume $n$ is sufficiently large and do not worry about specific stability ranges.

\subsection{Idea of proof} Consider the $\Aut(F_n)$-representation $H(n,1)$. For $q\ge 0$, we also consider it as a $A_n^{q+1}$-representation via the projection $A_n^{q+1}\to\Aut(F_n)$. The idea of the proof of the main theorem now has three main steps:\begin{enumerate}
    \item First, we show that as a $A_n^{q+1}$-representation, $H(n,q+1)$ is an extension of $H(n,1)$ by a trivial representation. We then apply Lemma \ref{lemma:extension}, together with Corollary \ref{cor:stable-cohomology-1} to compute the stable cohomology of $A_n^{q+1}$ with the coefficients $H(n,1)^{\otimes p}$, for any $p\ge 0$.
    \item\label{step2} Next, we consider the projection Hochschild-Serre spectral sequence associated to the short exact sequence
    $$1\to F_n^{\times q}\to A_n^{q+1}\to\Aut(F_n)\to 1,$$
    with the coefficients $H(n,1)^{\otimes p}$. We show that it degenerates on the $E_2$-page, which enables us to compute the stable cohomology of $\Aut(F_n)$ with coefficients in $\bigoplus_{i=0}^q H^i(F_n^{\times q})\otimes H(n,1)^{\otimes p}$.
    \item Finally we use the result from Step \eqref{step2} to inductively compute the desired stable cohomology groups.
\end{enumerate}

\subsection{Step 1: An extension of representations} We start by proving the following proposition:

\begin{proposition}
    We have a short exact sequence of $A_n^{q+1}$-representations
    \begin{equation}\label{eq:extension}
    0\to H(n,1)\to H(n,q+1)\to \bQ^{q}\to 0
\end{equation}
where $H(n,1)$ is considered an $A_n^{q+1}$-representation via the projection $A_n^{q+1}\to A_n^1$, and $\bQ^q$ is a trivial representation.
\end{proposition}

\begin{proof}
    The inclusion $R_{n,1}\hookrightarrow R_{n,q+1}$ induces an $A_n^{q+1}$-equivariant map $H(n,1)\to H(n,q+1)$. Furthermore, the cokernel in $H(n,q+1)$ is spanned by the homology classes of the paths from $x_1$ to each $x_i$, for $2\le i\le s+1$. By our description of the generators \eqref{AnsGen1}-\eqref{AnsGen4} of $A_n^{q+1}$ above, we see that the $A_n^{q+1}$-action on these classes is trivial modulo $H(n,1)$, which gives us the cokernel of the proposition
\end{proof}

\noindent We now apply Proposition \ref{prop:filtered-spectral-seq} to the short exact sequence \eqref{eq:extension}, which gives us, for any $p\ge 0$, a spectral sequence of the form
\begin{equation*}
    E_1^{k,l}=H^{k+l}\left(A_n^{q+1},\Ind_{\Sigma_{k}\times\Sigma_{p-k}}^{\Sigma_p}\left(H(n,q+1)^{\otimes k}\otimes(\bQ^q[1])^{\otimes(p-k)}\right)\right)\Rightarrow\  H^{k+l}\left(A_n^{q+1};H(n,1)^{\otimes p}\right)
\end{equation*}
We have
\begin{align*}
    E_1^{k,l}&=H^{k+l}\left(A_n^{q+1},\Ind_{\Sigma_{k}\times\Sigma_{p-k}}^{\Sigma_p}\left(H(n,q+1)^{\otimes k}\otimes(\bQ^q[1])^{\otimes(p-k)}\right)\right)\\
    &= H^{k+l}\left(A_n^{q+1}, \Ind_{\Sigma_k\times\Sigma_{p-k}}^{\Sigma_p}\left(H(n,q+1)^{\otimes k}\otimes(\bQ^q)^{\otimes(p-k)}\otimes\sgn_{p-k}\right)[p-k]\right) \\
    &= H^{l+2k-p}\left(A_n^{q+1}, \Ind_{\Sigma_k\times\Sigma_{p-k}}^{\Sigma_p}\left(H(n,q+1)^{\otimes k}\otimes(\bQ^q)^{\otimes(p-k)}\otimes\sgn_{p-k}\right)\right) \\
    &=\Ind_{\Sigma_k\times\Sigma_{p-k}}^{\Sigma_p}H^{l+2k-p}\left(A_n^{q+1}, H(n,q+1)^{\otimes k}\right) \otimes(\bQ^q)^{\otimes(p-k)}\otimes\sgn_{p-k}
\end{align*}
where we have used in the second step that as graded $\Sigma_{p-k}$-representations, we have 
$$(\bQ^q[1])^{\otimes(p-k)}\cong \left((\bQ^q)^{\otimes (p-k)}\otimes\sgn_{p-k}\right)[p-k],$$
by the Koszul sign rule, and in the last step that $A_n^{q+1}$ acts diagonally on the tensor powers and that $\bQ^q$ is a trivial $A_n^{q+1}$-representation. It follows by Corollary \ref{cor:stable-cohomology-1} that $E_1^{k,l}$ is only non-zero when $l+2k-p=k$, i.e\ when $k+l=p$, so the spectral sequence can have no non-zero differentials and thus it degenerates on the $E_1$-page. Thus, $H^*(A_n^{q+1},H(n,1)^{\otimes p})$ is concentrated in degree $p$ and in this degree, we get that
\begin{align*}
    H^p(A_n^{q+1},H(n,1)^{\otimes p})
    &\cong\bigoplus_{k=0}^p\Ind_{\Sigma_k\times\Sigma_{p-k}}^{\Sigma_p}\left( H^{k}\left(A_n^{q+1}, H(n,q+1)^{\otimes k}\right)\otimes(\bQ^q)^{\otimes(p-k)}\otimes\sgn_{p-k}\right)\\
    &\cong\bigoplus_{k=0}^p\Ind_{\Sigma_k\times\Sigma_{p-k}}^{\Sigma_p}\left((\cP_{k,0}\otimes\sgn_k)\otimes\overline{\cP}_{p-k,\le q}\otimes\sgn_{p-k}\right)\\
    &\cong \cP_{p,\le q}\otimes\sgn_p.
\end{align*}
where in the second step we have used Lemma \ref{lemma:partitions-tensor-prod} and Corollary \ref{cor:stable-cohomology-1} and in the final step we have used Lemma \ref{lemma:partitions-ind}\ref{partitions-ind2}.

In summary, we have the following lemma:

\begin{lemma}
For $n\gg *+p$, we have an isomorphism of $\Sigma_p\times\Sigma_q$-representations
\begin{align*}
    H^*\left(A_n^{q+1}; H(n,1)^{\otimes p}\right)\cong\begin{cases} \cP_{p,\le q}\otimes\sgn_{p}&\text{if }*=p,\\
    0&\text{if }*\neq p.\end{cases}
\end{align*}
\end{lemma}

\subsection{Step 2: Forgetting marked points}
Now consider the short exact sequence of groups
$$1\to F_n^{\times q}\to A_n^{q+1}\to\Aut(F_n)\to 1,$$
where the right hand map is the projection (topologically induced by forgetting the $q$ last marked points $x_2,\ldots,x_{q+1}$ of $R_{n,q+1}$) and the associated Hochschild-Serre spectral sequence with coefficients in the representation $H(n,1)^{\otimes p}$. As the $F_n^{\times q}$-action on this representation factors through $\Aut(F_n)$ via the short exact sequence, this is a trivial $F_n^{\times q}$-representation, so the spectral sequence has the form
\begin{equation}\label{ss:1}
\begin{gathered}
    E_{2}^{a,b}=H^a\left(\Aut(F_n), H^b(F_n^{\times q})\otimes H(n,1)^{\otimes p}\right)  \Rightarrow H^{a+b}\left(A_n^{q+1}, H(n,1)^{\otimes p}\right).
\end{gathered}
\end{equation}

\noindent Our next step is to show that this spectral sequence degenerates on the $E_2$-page, which is a special case of this more general lemma, which in fact holds over the integers:

\begin{lemma}\label{lem:degeneration-forgetting-marked-points}
    Let $M$ be a $\bZ[\Aut(F_n)]$-module. For any $s\ge 0$, the Hochschild-Serre spectral sequence with $M$-coefficients associated to the short exact sequence
    \begin{equation}\label{eq:ses-forgetting-marked-points-for-lemma}
        1\to F_n^{\times s}\to A_n^{s+1}\to\Aut(F_n)\to 1,
    \end{equation}
     degenerates on the $E_2$-page.
\end{lemma}

\begin{proof}
    For $s\ge 0$, let us denote the spectral sequence with $M$-coefficients associated to \eqref{eq:ses-forgetting-marked-points-for-lemma} by $E(s)$. We prove the lemma by induction on $s$. The case $s=0$ is trivial, so we suppose that for some $s\ge 0$ and each $0\le k<s$, $E(k)$ degenerates on the second page.

    As the $F_n^{\times q}$-action on $M$ is trivial, the terms of the page $E(s)_2$ have the form
    $$E(s)_2^{p,q}=H^p\left(\Aut(F_n),H^q\left(F_n^{\times s},\bZ\right)\otimes M\right),$$
    and thus this page is concentrated on the rows $0,\ldots, s$. We will show that the spectral sequence is a direct sum of spectral sequences concentrated on each of these rows.
    
    For $0\le k< s$, let $\Inj_{\le}(k,s)$ denote the set of order preserving injections from $[k]$ to $[s]$. For any $i\in \Inj_{\le}(k,s)$, we define an $\Aut(F_n)$-equivariant group homomorphism $\iota_i:F_n^{\times k}\to F_n^{\times s}$ (where $\Aut(F_n)$ acts diagonally on source and target) by
    \begin{align*}
        \iota_i(y_1,\ldots,y_k)_j=\begin{cases}
            y_l&\text{ if }i(l)=j\\
            1&\text{ otherwise}.
        \end{cases}
    \end{align*}
    We also define an $\Aut(F_n)$-equivariant group homomorphism $p_i:F_n^{\times s}\to F_n^{\times k}$ by
    \begin{align*}
        p_{i}(y_1,\ldots,y_s)=
            (y_{i(1)},\ldots,y_{i(k)}).
    \end{align*}
    Then we have $p_i\circ\iota_i=\id_{F_n^{\times k}}$.

    Since $\iota_i$ and $p_i$ are $\Aut(F_n)$-equivariant, the product maps
    \begin{align*}
      \id\times\iota_i&:\Aut(F_n)\ltimes F_n^{\times k}\to \Aut(F_n)\ltimes F_n^{\times s}\\
     \id\times p_i&:\Aut(F_n)\ltimes F_n^{\times s}\to \Aut(F_n)\ltimes F_n^{\times k}
    \end{align*}
    are well-defined. For any $i\in \Inj_{\le}(k,s)$, we thus get a commutative diagram 
    \[\begin{tikzcd}
        1\arrow[r]& F_n^{\times k}\arrow[d,hook,"{\iota}_i"]\arrow[r]&A_n^{k+1}\arrow[d,hook,"\id\times \iota_i"]\arrow[r]&\Aut(F_n)\arrow[r]\arrow[d,-,double equal sign distance,double]&1\\
        1\arrow[r]& F_n^{\times s}\arrow[d,"p_i",two heads]\arrow[r]&A_n^{s+1}\arrow[d,"\id\times p_i", two heads]\arrow[r]&\Aut(F_n)\arrow[d,-,double equal sign distance,double]\arrow[r]&1\\
        1\arrow[r]& F_n^{\times k}\arrow[r]&A_n^{k+1}\arrow[r]&\Aut(F_n)\arrow[r]&1,
    \end{tikzcd}\]
    which defines a split monomorphism of short exact sequences. By the functoriality of the Hochschild-Serre spectral sequence with respect to morphisms of short exact sequences, we thus get an induced commutative diagram 
    \begin{equation}\label{eq:split-sseq}
        \begin{tikzcd}
        E(k)\arrow[r,"\iota_i"]\arrow[rd,"\id"]&E(s)\arrow[d,"p_i"]\\
        & E(k)
    \end{tikzcd}
    \end{equation}
    of maps of spectral sequences, where we denote the induced maps by the same symbols, for brevity.

    For $k<s$, let $\tilde{E}(k)$ denote the trivial spectral sequence defined by
    \begin{align*}
        \tilde{E}(k)_r^{p,q}:=\begin{cases}
            E(k)_2^{p,q}&\text{ if }q=k\\
            0&\text{ otherwise}.
        \end{cases}
    \end{align*}
    and having no nonzero differentials. By our inductive assumption that $E(k)$ degenerates on the second page, the spectral sequence $\tilde{E}(k)$ is a direct summand of $E(k)$, so there is a commutative diagram
    \begin{equation*}
        \begin{tikzcd}
        \tilde{E}(k)\arrow[r]\arrow[rd,"\id"]&E(k)\arrow[d]\\
        & \tilde{E}(k)
    \end{tikzcd}
    \end{equation*}
    of maps of spectral sequences. Composing with the diagram \eqref{eq:split-sseq}, we thus get a commutative diagram
    \begin{equation}
        \begin{tikzcd}
        \tilde{E}(k)\arrow[r,"\tilde{\iota}_i"]\arrow[rd,"\id"]&E(s)\arrow[d,"\tilde{p}_i"]\\
        & \tilde{E}(k)
    \end{tikzcd}
    \end{equation}
    of maps of spectral sequences.

    Iteratively using the Künneth formula and the fact that $F_n$ is free, so that its cohomology is a free $\bZ$-module in each degree, we get that for any $0\le k\le s$, we have an isomorphism
    $$H^k(F_n^{\times s},\bZ)\cong \bigoplus_{\substack{I\subset[s]\\|I|=k}} H^1(F_n,\bZ)^{\otimes I}$$
    of $\bZ[\Aut(F_n)]$-modules. Looking at the $k$th row of the second page of each spectral sequence, we thus have
    $$E(s)_2^{p,k}=H^p\left(\Aut(F_n);H^k\left(F_n^{\times s},\bZ\right)\otimes M\right)\cong \bigoplus_{\substack{I\subseteq[s]\\|I|=k}}H^p\left(\Aut(F_n);H^1(F_n,\bZ)^{\otimes I}\otimes M\right)$$
    and 
    $$\tilde{E}(k)_2^{p,k}=E(k)_2^{p,k}=H^p\left(\Aut(F_n);H^1\left(F_n,\bZ\right)^{\otimes k}\otimes M\right).$$    
    The map 
    $$({\tilde{\iota}}_i)_2^{p,k}:H^p\left(\Aut(F_n);H^1(F_n,\bZ)^{\otimes k}\otimes M\right)\to\bigoplus_{\substack{I\subseteq[s]\\|I|=k}} H^p\left(\Aut(F_n);H^1\left(F_n,\bZ\right)^{\otimes I}\otimes M\right)$$
    is the map landing in the summand of $\Image(i)\subseteq[s]$, induced by the map of coefficients 
    $$H^1(F_n,\bZ)^{\otimes k}\to H^1(F_n,\bZ)^{\otimes\Image(i)}$$
    defined by $i\in \Inj_{\le}(k,s)$. The map $\Inj_{\le}(k,s)\to \{I\subseteq[s]\mid|I|=k\}$, given by $i\mapsto\Image(i)$, is a bijection, so we have an isomorphism
    $$\bigoplus_{\iota\in\Inj_{\le}(k,s)}(\tilde{\iota})_2^{p,k}:\bigoplus_{\iota\in\Inj_{\le}(k,s)} \tilde{E}(k)_2^{p,k}\to E(s)_2^{p,k}.$$
    In total, we thus get a commutative diagram
    \begin{equation}
        \begin{tikzcd}
        \bigoplus_{\iota\in\Inj_{\le}(k,s)}\tilde{E}(k)\arrow[rr,"\bigoplus_{\iota\in\Inj_{\le}(k,s)}\tilde{\iota}_i"]\arrow[rrdd,"\id"]&&E(s)\arrow[dd,"\bigoplus_{\iota\in\Inj_{\le}(k,s)}\tilde{p}_i"]\\ \\
        & &\bigoplus_{\iota\in\Inj_{\le}(k,s)}\tilde{E}(k),
    \end{tikzcd}
    \end{equation}
    where each map is an isomorphism when restricted to the $k$th row of every page. Hence $E(s)$ is the direct sum of a spectral sequence concentrated on the $k$th row and one whose terms are all zero on this row. Since this holds for all $0\le k<s$, it follows that $E(s)$ is a direct sum of spectral sequences concentrated on distinct rows $0,\ldots,s-1$ as well as one which is zero below the $s$th row. However, since $E(s)$ is concentrated on the rows $0,\ldots,s$, it follows that the last summand is also concentrated on a single row. Thus, the spectral sequence must degenerate on $E(s)_2$.
\end{proof}

\noindent Combining the previous two lemmas, we get the following corollary:

\begin{corollary}\label{corollary:step2}
For $n\gg *+p+q$ and $0\le k\le q$ we have 
\begin{equation}\label{eq:vanishing}
    H^*\left(\Aut(F_n);H^k(F_n^{\times q})\otimes H(n,1)^{\otimes p}\right)=0
\end{equation}
if $*\neq p-k$ and
\begin{equation}\label{eq:nonzero-cohomology}
\begin{gathered}
   \bigoplus_{k=0}^pH^{p-k}\left(\Aut(F_n);H^k(F_n^{\times q})\otimes H(n,1)^{\otimes p}\right)  \cong \cP_{p,\le q}\otimes\sgn_p
\end{gathered}
\end{equation}
as $\Sigma_p\times\Sigma_q$-representations.
\end{corollary}

\subsection{Step 3: Proof of Theorem A}

Now let us finish the proof of the main theorem. 

\begin{proof}[Proof of Theorem A]
Since $H^q(F_n^{\times q})\cong H^*(n,1)^{\otimes q}$ as $\Aut(F_n)$-representations, the vanishing part of the theorem follows directly from the first part of Corollary \ref{corollary:step2}. 

Now let us assume that $q\le p$. From the second part of Corollary \ref{corollary:step2}, together with Lemma \ref{lemma:partitions-ind}\ref{partitions-ind1}, we get that 
\begin{equation}\label{eq:forinduction}
    \bigoplus_{k=0}^qH^{p-k}\left(\Aut(F_n);H^k(F_n^{\times q})\otimes H(n,1)^{\otimes p}\right)\cong \bigoplus_{k=0}^q\Ind_{\Sigma_k\times\Sigma_{q-k}}^{\Sigma_q}\cP_{p,k}\otimes\sgn_p.
\end{equation}
We will now use induction on $q$ to prove that
$$H^{p-q}\left(\Aut(F_n); H^{q}(F_n^{\times q})\otimes H(n,1)^{\otimes p}\right)\cong \cP_{p,q}\otimes\sgn_p,$$
as $\Sigma_p\times\Sigma_q$-representations. For $q=0$, this statement is just \cite[Theorem A(ii)]{RW18}, so assume that 
$$H^{p-k}(\Aut(F_n);H^k(F_n^{\times k})\otimes H(n,1)^{\otimes p})\cong\cP_{p,k}\otimes\sgn_p$$
as $\Sigma_{p}\times\Sigma_i$-representations, for all $k<q$. We have
$$H^k(F_n^{\times q})\cong \bigoplus_{\substack{I\subseteq[q]\\|I|=k}} H^k\left(F_n^{\times I}\right)$$
as $\Aut(F_n)$-representations by the Künneth theorem, and using Equation \eqref{eq:inductionaction} from above, we see that as $\Sigma_q$-representations, we have 
$$H^k(F_n^{\times q})\cong\Ind_{\Sigma_k\times\Sigma_{q-k}}^{\Sigma_q} H^k(F_n^{\times k}),$$
where we consider $H^k(F_n^{\times k})$ as a $\Sigma_k\times\Sigma_{q-k}$-representation via the projection $\Sigma_k\times\Sigma_{q-k}\to\Sigma_k$. This gives us an isomorphism
$$\bigoplus_{k=0}^{q-1} H^{p-k}\left(\Aut(F_n); H^k(F_n^{\times q})\otimes H(n,1)^{\otimes p}\right)\cong\bigoplus_{k=0}^{q-1}\Ind_{\Sigma_k\times\Sigma_{q-k}}^{\Sigma_q}\cP_{p,k}\otimes\sgn_p.$$
of $\Sigma_p\times \Sigma_q$-representations. Since $H^q(F_n^{\times q})\cong H^1(F_n)^{\otimes q}\otimes\sgn_q$ as $\Sigma_q$-representations by the Künneth theorem (where the sign representation appears because of the Koszul sign rule), the theorem now follows by combining this with Equation \eqref{eq:forinduction}.
\end{proof}

\section{Homological stability with finite degree coefficients}\label{sec:stability}

\noindent In this section we prove Theorem \ref{thm:twostabilizations}, using a general framework for homological stability that was introduced by Randal-Williams and Wahl in \cite{WahlRW} and generalized by Krannich in \cite{KrannichStability}. We will assume some familiarity with these tools, and only briefly summarize the necessary background, in the case where we start from a symmetric monoidal groupoid, since this is enough for our purposes and somewhat simplifies the background. Wahl \cite{WahlTool} has also written a survey on this framework, which may be useful to the unfamiliar reader. 

\subsection{A framework for homological stability}\label{sec:framework-homstab} Let $(\cG,\oplus,0)$ be a monoidal groupoid.

\subsubsection{Quillen's bracket construction} Let us recall the definition of the category $U\cG$ associated to $\cG$, constructed using Quillen's bracket construction (cf. \cite[Section 1.1]{WahlRW}):

\begin{definition}
    Let $(\cG,\oplus, 0)$ be a monoidal groupoid. The category $U\cG$ has the same objects as $\cG$, but a morphism $A\to B$ in $U\cG$ is an equivalence class of pairs $(C,f)$, where $C\in\cG$ and $f: C\oplus A\to B$ is a morphism of $\cG$, with two such pairs $(C,f)$ and $(C',f')$ are considered equivalent whenever there is a morphism $g\in\Hom_{\cG}(C,C')$ such that the diagram
    \[\begin{tikzcd}
        C\oplus A\arrow[r,"f"]\arrow[d," g\oplus\id_A"']&B\\
        C'\oplus A\arrow[ru,"f'"']
    \end{tikzcd}\]
    is commutative. We write $[C,f]$ for the equivalence class of the pair $(C,f)$.
\end{definition}

\begin{proposition}[{\cite[Proposition 1.8]{WahlRW}}]\label{prop:mononidal-structure-on-UG}
    If $(\cG,\oplus,0)$ is a symmetric monoidal groupoid, then $U\cG$ has a symmetric monoidal structure, with the monoidal product defined the same way as that of $\cG$ on objects, and for $[X,f]\in\Hom_{U\cG}(A,B)$ and $[Y,g]\in\Hom_{U\cG}(C,D)$ defined by
    $$[X,f]\oplus[Y,g]=[X\oplus Y,(f\oplus g)\circ (\id_X\oplus s_{A,Y}^{-1}\oplus\id_C)]\in\Hom_{U\cG}(A\oplus C,B\oplus D),$$
    where $s$ denotes the symmetry of $\cG$. The symmetry is given by $[0,s_{A,B}]:A\oplus B\to B\oplus A$, for any $A,B\in U\cG$. Furthermore, the monoidal unit $0$ is initial in $U\cG$.
\end{proposition}

\begin{notation}
 For any $X\in U\cG$, we write $\iota_X:=[X,\mathrm{id}_X]:0\to X$ for the unique morphism in $\Hom_{U\cG}(0,X)$. Furthermore, for $A,B\in U\cG$, by abuse we will write $s_{A,B}:=[0,s_{A,B}]$ for the symmetry, for simplicity.
\end{notation}

\noindent For $A,X\in\cG$, we define the group $G_n:=\Aut_{U\cG}\left(A\oplus X^{\oplus n}\right)$ and note that we have a stabilization map $-\oplus \id_X:G_n\to G_{n+1}$. We thus get a sequence of groups for which we can consider homological stability. For the main theorem from \cite{WahlRW} about homological stability for this sequence to hold, we will also need to assume that $\cG$ and the pair $(A,X)$ satisfies the following property:

\begin{definition}
    For $A,X\in\cG$, we say that $\cG$ satisfies local cancellation at $(A,X)$ if for all $0\le p<n$, if $Y\in\cG$ is such that $Y\oplus X^{\oplus p+1}\cong A\oplus X^{\oplus n}$, then $Y\cong A\oplus X^{\oplus n-p-1}$.
\end{definition}

\noindent Furthermore, we will need to assume that the stabilization maps satisfy the following:

\begin{definition}
    We say that $U\cG$ satisfies \textit{injectivity} at $(A,X)$ if for all $0\le p<n$, the stabilization map $-\oplus\id_{X^{\oplus p+1}}: G_{n-p-1}\to G_{n}$ is injective.
\end{definition}

\noindent We define our symmetric monoidal groupoids below so that the automorphism groups of their objects are the groups for which we want to prove homological stability, i.e. $G_n$. However, for $\cG$ a symmetric monoidal groupoid as above, it is not immediate that $G_n=\Aut_{U\cG}(A)=\Aut_{\cG}(A)$, for any $A\in \cG$, but this holds under quite mild assumptions:

\begin{proposition}[{\cite[Proposition 1.7]{WahlRW}}]\label{prop:underlying-groupoid}
    Let $(\cG,\oplus, 0)$ be a monoidal category such that $\Aut_{\cG}(0)=\{\id_0\}$ and which has no zero-divisors, i.e.\ if $A\oplus B\cong 0$ in $\cG$, then $A\cong B\cong 0$. Then $\cG$ is the underlying groupoid of $U\cG$.
\end{proposition}

\subsubsection{Coefficient systems} Now let us recall the notion of coefficient systems. Let $\cC_{A,X}$ be the full subcategory of $U\cG$ with objects of the form $A\oplus X^{\oplus n}$, for $n\ge 0$. We call a functor $F:\cC_{A,X}\to\Ab$, where $\Ab$ is the category of abelian groups, a \textit{coefficient system} on $\cC_{A,X}$. 

In order to define stabilization maps of coefficients, we recall from \cite[Section 4.1]{WahlRW} the definition of the \textit{upper suspension} functor.

\begin{definition}
    We define the upper suspension functor $\Sigma^X:\cC_{A,X}\to\cC_{A,X}$, by $\Sigma^X(B)=B\oplus X$ and 
    $$\Sigma^X(f:B\to B')=f\oplus\id_X:B\oplus X\to B'\oplus X.$$
    For $B\in \cC_{A,X}$, we further define the \textit{upper suspension map}
    $$\sigma^X_B:=\id_{B}\oplus\iota_X:B\to B\oplus X.$$
    This defines the components of a natural transformation $\sigma^X:\id\to\Sigma^X$. For brevity, when $B=A\oplus X^{\oplus n}$ we will write $\sigma_n^X:=\sigma_B^X$.
\end{definition}

\noindent Given a coefficient system $F:\cC_{A,X}\to\Ab$, we define a $G_n$-module $F_n:=F\left(A\oplus X^{\oplus n}\right)$. Furthermore, the suspension map $\sigma^X_{n}:A\oplus X^{\oplus n}\to A\oplus X^{\oplus n+1}$ induces a $G_n$-equivariant map $F_n\to F_{n+1}$. We thus get an induced sequence in homology
\begin{equation}\label{eq:coefficient-stabilization}
    \left(-\oplus\id_X,F(\sigma_n^X)\right)_*:H_*(G_n,F_n)\to H_*(G_{n+1},F_{n+1}),
\end{equation}
for which we can consider homological stability. 

In the cases we consider below, we already have groups $G_n$ and coefficients $F_n$ together with stabilization maps, and will use the following proposition to construct coefficient systems from this data:

\begin{proposition}\label{prop:extendcoefficientstofunctor}
    Let $(\cG,\oplus,0)$ be a symmetric monoidal groupoid with no zero divisors and such that $\Aut_{\cG}(0)=\{\id_0\}$. Let $(A,X)$ be a pair of objects in $\cG$ such that\begin{enumerate}[(i)]
        \item\label{cond-i:extend-coeff-system} in $U\cG$ we have $A\oplus X^{\oplus n}\not\cong A\oplus X^{\oplus m}$ when $n\neq m$,
        \item\label{cond-ii:extend-coeff-system} $\cG$ satisfies local cancellation at $(A,X)$ and
        \item\label{cond-iii:extend-coeff-system} $U\cG$ satisfies injectivity at $(A,X)$. 
    \end{enumerate}
    Let $G_n=\Aut_{U\cG}(A\oplus X^{\oplus n})$ and suppose that for $n\ge 0$, $F_n$ is a sequence of $G_n$-modules, and $\sigma_n:F_n\to F_{n+1}$ is a sequence of $G_n$-equivariant maps such that 
    \begin{enumerate}[(i)]
      \setcounter{enumi}{3}
        \item\label{cond-iv:extend-coeff-system} $\Aut_{U\cG}(X^{\oplus m})$ acts trivially on the image of $F_n$ in $F_{n+m}$. 
    \end{enumerate}
    Then there exists a functor
    $$F:\cC_{A,X}\to\Ab$$
    satisfying that $F(A\oplus X^{\oplus n})=F_n$ that $\sigma_n=F(\sigma_n^X)$. For $f:A\oplus X^{\oplus n}\to A\oplus X^{\oplus m}$, with $n\le m$, $F(f)$ is given by writing $f$ as
    $$f=\phi\circ(\id_{A\oplus X^{\oplus n}}\oplus \iota_{X^{\oplus m-n}}),$$
    for $\phi\in G_m=\Aut_{U\cG}(A\oplus X^{\oplus m})$.
\end{proposition}

\begin{proof}
  By \cite[Theorem 1.10]{WahlRW}, the assumptions that $\cG$ has no zero-divisors and that local cancellation and injectivity hold at $(A,X)$ imply that $\cC_{A,X}$ is a \textit{pre-braided homogeneous category}. Now the proposition follows from \cite[Proposition 4.2]{WahlRW}.
\end{proof}

\noindent Homological stability is not enjoyed by arbitrary coefficient systems, so we will limit ourselves to considering coefficient systems with an additional property, namely having \textit{finite degree}. In order to define this property, we first recall from \cite[Section 4.2]{WahlRW} the definition of the \textit{lower} suspension functor. For $A,B$ in $U\cG$, recall that we write $s_{A,B}:A\oplus B\to B\oplus A$ for the symmetry.

\begin{definition}
We define the lower suspension functor $\Sigma_X:\cC_{A,X}\to\cC_{A,X}$ by $\Sigma_X(B)=B\oplus X$ and for $f:A\oplus X^{\oplus n}\to A\oplus X^{\oplus k}$, $\Sigma_X(f)$ is the composition
 \[\begin{tikzcd}
     A\oplus X^{\oplus n+1}\arrow[rr,"s_{A,X}\oplus \id"]&&X\oplus A\oplus X^{\oplus n}\arrow[r,"\id\oplus f"]&X\oplus A\oplus X^{\oplus k}\arrow[rr,"s_{X,A}\oplus\id"]&&A\oplus X^{\oplus k+1}.
 \end{tikzcd}\]
  We also define the \textit{lower suspension maps}
 $$\sigma_X^n:=(s_{X,A}\oplus X^{\oplus n})\circ(\iota_X\oplus\id):A\oplus X^{\oplus n}\to A\oplus X\oplus X^{\oplus n},$$
 which define a natural transformation $\sigma_X:\id\to\Sigma_X$. 
\end{definition}

 \begin{remark}\label{remark:upperlowersusp}
     It is explained in \cite[Section 4.2]{WahlRW} how the upper and lower suspensions are related via the symmetric monoidal structure. In particular, we have
     $$\sigma_X^n=(s_{X,A}\oplus \id_{X^{\oplus n}})\circ s_{A\oplus X^{\oplus n},X}\circ\sigma^X_n.$$
 \end{remark}

\begin{definition}
    If $F:\cC_{A,X}\to\Ab$ is a coefficient system, we define the suspension of $F$ as the composite functor
 $$\tau_X(F):=F\circ\Sigma_X.$$
Furthermore, we have a natural transformation $i_X(F):=F(\sigma_X):F\to\tau_X(F)$ and we define $\kappa_X(F):=\ker(F)$ and $\delta_X(F):=\coker(F)$ as the kernel and cokernel, respectively, of this natural transformation. Furthermore, for each $r\ge 1$ we write $\delta_X^r(F)$ for the iterated application of $\delta_X$. We also use the convention that $\delta_X^0(F)=F$.
\end{definition}  

\noindent With these definitions, we can now define the notion of a finite degree coefficient system:

 \begin{definition}[{\cite[Definition 4.10]{WahlRW}}]
     A coefficient system $F:\cC_{A,X}\to\Ab$ is of degree $r<0$ at $N\in\bZ$ if $F(A\oplus X^{\oplus n})=0$ for all $n\ge N$. For $r\ge 0$, we define inductively that $F$ has degree $r$ at $N\in\bZ$ if
    \begin{enumerate}[(i)]
         \item\label{cond1:degree} $\kappa_X(F)$ has degree $-1$ at $N$,
         \item\label{cond2:degree} $\delta_X(F)$ has degree $r-1$ at $N-1$.
     \end{enumerate}
 \end{definition}

\noindent The coefficient systems we consider below are given by composing a coefficient system of degree 1 with the functor $(-)^{\otimes p}:\Ab\to\Ab$ and we therefore need to understand how degrees of coefficient systems behave under taking tensor products. A proposition explaining this appears in \cite{ExtCorrigendumSoulié}, but we include its proof here for completeness. Before stating the proposition, we introduce the following assumptions, which are used in the proofs of Proposition \ref{prop:tensorproductcoefficients} and its corollary, respectively:

\begin{assumption}\label{ass:subcategory-free-modules}
Let $F:\cC_{A,X}\to \Ab$ be a coefficient system. For each $r\geq 0$, the functor $\delta^{r}_{X}(F) \colon \cC_{A,X} \to \Ab$\begin{enumerate}[(a)]
    \item\label{ass:subcategory-free-modules(a)} factors through the subcategory of free abelian groups,
    \item\label{ass:subcategory-free-modules(b)} factors through the subcategory of finitely generated free abelian groups.
\end{enumerate} 
If there is some $N\ge 0$ such that either of the assumptions hold only for the values $F(A\oplus X^{\oplus n})$ for $n\ge N$, we say that the assumption is satisfied at $N$.
\end{assumption}

\noindent We can now slightly restate \cite[Proposition 0.8]{ExtCorrigendumSoulié} as follows:

\begin{proposition}\label{prop:tensorproductcoefficients}
    Let $F,G:\cC_{A,X}\to \Ab$ be coefficient systems of finite degrees $d\geq0$ at $N$ and $d'\geq0$ at $N'$ respectively and satisfying Assumption~\ref{ass:subcategory-free-modules}\eqref{ass:subcategory-free-modules(a)} at $N$ and $N'$ respectively. We denote by $M$ the maximum $\mathrm{max}(N,N')$. Then the functor $F\otimes G$ is a coefficient system of finite degree less or equal to $d+d'$ at $M+\frac{(d+d'+1)(d+d'+2)}{2}$.
\end{proposition}

\begin{remark}
    The value $M+\frac{(d+d'+1)(d+d'+2)}{2}$ of Proposition \ref{prop:tensorproductcoefficients} is likely not optimal (indeed, the tensor product of two systems of finite degree at zero ought to be of finite degree at zero), but is rather an artefact of the specific proof. Since we do not use the explicit range of stability obtained in the results of this section, we will typically omit this specific value going forward, and rather use that if $F$ is of finite degree at integer $N$, then $F^{\otimes p}$ is of finite degree $r$ at \textit{some} $N'$, depending on $p,r$ and $N$.
\end{remark}

\noindent To prove Proposition \ref{prop:tensorproductcoefficients}, we need the following lemma.

\begin{lemma}\label{lemma:exact-seq-coefficient-systems}
    Let $F,G:\cC_{A,X}\to \Ab$ be coefficient systems satisfying Assumption \ref{ass:subcategory-free-modules}\eqref{ass:subcategory-free-modules(a)}. There is an exact sequence in the functor category $\Fun(\cC_{A,X},\Ab)$:
\begin{equation}\label{eq:exact-sequence-kappa-delta-tensor-free}
F\otimes \kappa_{X}(G) \hookrightarrow \kappa_{X}(F\otimes G) \to \kappa_{X}(F)\otimes \tau_{X}(G) \to F\otimes \delta_{X}(G) \to \delta_{X}(F\otimes G) \twoheadrightarrow \delta_{X}(F)\otimes \tau_{X}(G).
\end{equation}
\end{lemma}

\begin{proof}
We follow the proofs of \cite[Lemma 0.1, Corollary 0.5]{ExtCorrigendumSoulié}. By definition of $i_X$, we have $i_X(F\otimes G)=i_X(F)\otimes i_X(G)$, so it factors as
$$(i_{X}(F)\otimes \id_{\tau_{X}(G)})\circ (\id_{F}\otimes i_{X}(G))\colon F\otimes G\to F\otimes \tau_{X}(G)  \to \tau_{X}(F\otimes G).$$
This gives us the commutative diagrams of functors
\begin{equation}\label{diag:composition-tensor-cokernel}
\begin{tikzcd}
F\otimes G \arrow[rrr,"\id_{F}\otimes i_{X}(G)"]\arrow[dd,"i_{X}(F\otimes G)"]
&&& F\otimes \tau_X(G) \arrow[dd,"i_{X}(F)\otimes \id_{\tau_{X}(G)}"]\arrow[rrr, two heads]
&&&  \coker(\id_{F}\otimes i_{X}(G)) \arrow[dd]
\\\\
\tau_X(F\otimes G)\arrow[rrr,equal]
&&&\tau_{X}(F\otimes G)\arrow[rrr]
&&& 0
\end{tikzcd}
\end{equation}
and
\begin{equation}\label{diag:composition-tensor-kernel}
\begin{tikzcd}
0 \arrow[rrr,hook] \arrow[dd,hook]
&&& F\otimes G \arrow[rrr,equal]\arrow[dd,"\id_{F}\otimes i_{X}(G)"]
&&& F\otimes G \arrow[dd,"i_{X}(F\otimes G)"]
\\\\
\ker(i_{X}(F)\otimes \id_{\tau_{X}(G)}) \arrow[rrr,hook]
&&& F\otimes \tau_{X}(G)\arrow[rrr,"i_{X}(F)\otimes \id_{\tau_{X}(G)}"]
&&&\tau_{X}(F\otimes G),
\end{tikzcd}
\end{equation}
both of which have exact rows. We have $\coker(\id_{F}\otimes i_{X}(G))\cong F\otimes \delta_{X}(G)$ and $\coker(i_{X}(F)\otimes \id_{\tau_{X}(G)})\cong \delta_{X}(F)\otimes \tau_{X}(G)$ by the right-exactness of the tensor product. Since $F$ and $\tau_{X}G$ take values in free (and thus flat) abelian groups, by assumption, it follows that $\ker(\id_{F}\otimes i_{X}(G))\cong F\otimes \kappa_{X}(G)$ and $\ker(i_{X}(F)\otimes \id_{\tau_{X}(G)}) \cong \kappa_{X}(F)\otimes \tau_{X}(G)$. Applying the snake lemma to \eqref{diag:composition-tensor-cokernel} and \eqref{diag:composition-tensor-kernel}, we obtain exact sequences
\begin{equation}\label{eq:snake-composition-tensor-cokernel}
\kappa_X(F\otimes G) \to
  \ker(i_X(F)\otimes \id_{\tau_X(G)})\to
F\otimes\delta_X(G)\to\delta_{X}(F\otimes G)\twoheadrightarrow
\delta_{X}(F)\otimes \tau_{X}(G)
\end{equation}
and
\begin{equation}\label{eq:snake-composition-tensor-kernel}
\ker(\id_F\otimes i_X(G))\hookrightarrow\kappa_X(F\otimes G)\to
\ker(i_X(F)\otimes \id_{\tau_X(G)}) \to F\otimes\delta_X(G)\to
\delta_{X}(F\otimes G).
\end{equation}
Combining \eqref{eq:snake-composition-tensor-cokernel} and \eqref{eq:snake-composition-tensor-kernel} gives us the desired exact sequence.
\end{proof}

\begin{proof}[Proof of Proposition \ref{prop:tensorproductcoefficients}]
    We follow the proof of \cite[Proposition 0.8]{ExtCorrigendumSoulié}. For brevity, whenever $F:\cC_{A,X}\to\Ab$ is a coefficient system, we will write $F(n):=F(A\oplus X^{\oplus n})$.
    
    The proof is done by strong induction on the sum of the degrees $d+d'\geq 0$. If $d+d'=0$, then it follows from the exact sequence \eqref{eq:exact-sequence-kappa-delta-tensor-free} of Lemma~\ref{lemma:exact-seq-coefficient-systems} that $\kappa_{X}(F\otimes G)(n)=0$ and $\delta_{X}(F\otimes G)(n+1)=0$ for all $n\geq M$, so $F\otimes G$ is a coefficient system of degree $0$ at $M+1$.

    Assume now that the claim holds for all $d,d'$ with $0 \leq d,d'\leq D-1$ for some $D\geq 1$, and now consider the case of $d+d'=D$. To show that $F\otimes G$ is of degree at most $d+d'$ at $M+\frac{(D+1)(D+2)}{2}$ amounts to showing that for each $0\le r\le D$ and $n\ge M+\frac{(D+1)(D+2)}{2}-r$ we have $\kappa_X(\delta_X^r(F\otimes G))(n)=0$ and that  for $n\ge M+\frac{(D+1)(D+2)}{2}-D-1$ we have $\delta^{D+1}(F\otimes G)(n)=0$.
    
     It is immediate from the definition that $\tau_X$ preserves degrees of coefficient systems, so $\tau_X(G)$ is a coefficient system of degree $d'$ at $N'$. It thus follows from the inductive assumption and the exact sequence \eqref{eq:exact-sequence-kappa-delta-tensor-free} of Lemma \ref{lemma:exact-seq-coefficient-systems} that $\kappa_X(F\otimes G)(n)=0$ for all $n\ge M$. Furthermore, for $n\ge M$ we have a short exact sequence
     \begin{equation}\label{eq:ses-difference-functors}
        \begin{tikzcd}
            F\otimes \delta_{X}(G)(n) \arrow[r,hook]&
\delta_{X}(F\otimes G)(n)  \arrow[r, two heads]
& (\delta_X(F)\otimes\tau_X(G))(n)  
        \end{tikzcd}
    \end{equation}
    of $\bZ[G_n]$-modules. By exactness of the translation functor $\tau_X$, applying the snake lemma gives us, for each $n\ge M$, an exact sequence
    \begin{equation}\label{eq:snake-applied-to-suspension}
        \begin{tikzcd}
            \kappa_X(F\otimes\delta_X(G))(n) \arrow[r,hook]
& \kappa_{X}(\delta_X(F\otimes G))(n) \arrow[r]
\arrow[d, phantom, ""{coordinate, name=Z}]
& \kappa_X(\delta_X(F)\otimes\tau_X(G))(n) \arrow[dll,
rounded corners,
to path={ -- ([xshift=2ex]\tikztostart.east)
|- (Z) [near end]\tikztonodes
-| ([xshift=-2ex]\tikztotarget.west)
-- (\tikztotarget)}] \\
\delta_X(F\otimes \delta_{X}(G))(n) \arrow[r]&
\delta_{X}^2(F\otimes G)(n)  \arrow[r, two heads]
& \delta_{X}(\delta_X(F)\otimes\tau_X(G))(n)        
    \end{tikzcd}
    \end{equation}
    of $\bZ[G_n]$-modules. By assumption, $\delta_X(F)$ and $\delta_X(G)$ are, respectively, of degrees $d-1$ at $N-1$ and $d'-1$ at $N'-1$, so by Assumption \ref{ass:subcategory-free-modules}\eqref{ass:subcategory-free-modules(a)} and the inductive assumption it follows that both $F\otimes\delta_X(G)$ and $\delta_X(F)\otimes\tau_X(G)$ are of degree at most $d+d'-1=D-1$ at $M+\frac{D(D+1)}{2}$. We thus deduce from the exact sequence \eqref{eq:snake-applied-to-suspension} that $\kappa_X(\delta_X(F\otimes G))(n)=0$ for $n\ge M+\frac{D(D+1)}{2}$ (in particular, for $n\ge M+\frac{(D+1)(D+2)}{2}-1)$ and that for such $n$ we have a short exact sequence
    \begin{equation}\label{eq:ses-from-snake-applied-to-suspension}
        \begin{tikzcd}
            \delta_X(F\otimes \delta_{X}(G))(n) \arrow[r,hook]&
\delta_{X}^2(F\otimes G)(n)  \arrow[r, two heads]
& \delta_{X}(\delta_X(F)\otimes\tau_X(G))(n)  
        \end{tikzcd}
    \end{equation}
    of $\bZ[G_n]$-modules. We can now apply the same argument that we applied to the short exact sequence \eqref{eq:ses-difference-functors} to the short exact sequence \eqref{eq:ses-from-snake-applied-to-suspension}. By induction on $2\le r\le D$, it follows that $\kappa_X(\delta_X^r(F\otimes G))(n)=0$ for $n\ge M+\frac{D(D+1)}{2}$ (in particular, for $n\ge M+\frac{(D+1)(D+2)}{2}-r$) and for such $n$, we finally we end up with a short exact sequence
    \begin{equation}\label{eq:final-ses-from-snake-applied-to-suspension}
        \begin{tikzcd}
            \delta_X^D(F\otimes \delta_{X}(G))(n) \arrow[r,hook]&
\delta_{X}^{D+1}(F\otimes G)(n)  \arrow[r, two heads]
& \delta_{X}^D(\delta_X(F)\otimes\tau_X(G))(n)  
        \end{tikzcd}
    \end{equation}
    of $\bZ[G_n]$-modules. By the inductive assumption, we have that both $\delta_X^D(F\otimes \delta_{X}(G))(n)=0$ and $ \delta_{X}^D(\delta_X(F)\otimes\tau_X(G))(n)=0$ for $n\ge M+\frac{D(D+1)}{2}-D$, so we conclude that $\delta_X^{D+1}(F\otimes G)(n)=0$ for all $n\ge M+\frac{D(D+1)}{2}=M+\frac{(D+1)(D+2)}{2}-D-1$. 
\end{proof}

\noindent As we want to apply Proposition \ref{prop:tensorproductcoefficients} to tensor powers of coefficient systems, we deduce the following corollary:

\begin{corollary}[{\cite[Corollary 0.9]{ExtCorrigendumSoulié}}]\label{cor:tensorpowercoefficients}
    Suppose that $F:\cC_{A,X}\to\Ab$ is a coefficient system of degree $r$ at $N\in\bZ$, satisfying Assumption \ref{ass:subcategory-free-modules}\eqref{ass:subcategory-free-modules(b)} at $N$. For any $p\ge 0$, the coefficient system $F^{\otimes p}:\cC_{A,X}\to\Ab$ is a coefficient system of degree at most $p\cdot r$ at $N+\frac{1}{2}\sum_{k=2}^p(kr+1)(kr+2)$.
\end{corollary}

\begin{proof}
    We follow the proof of \cite[Corollary 0.9]{ExtCorrigendumSoulié}, which is by strong induction on $p$. For $p=2$, the statement follows immediately from Proposition \ref{prop:tensorproductcoefficients}. Now assume the result holds for all $2\le p\le P$ for some $P\ge 2$. As the category of finitely generated free abelian groups is closed under extensions, it follows from the short exact sequences \eqref{eq:snake-applied-to-suspension}, \eqref{eq:ses-from-snake-applied-to-suspension} and so on until \eqref{eq:final-ses-from-snake-applied-to-suspension}, with $G=F$, that for each $r\ge 0$ and $n\ge N$, $\delta_X^r(F^{\otimes p})(n)$ is a finitely generated free abelian group. Iteratively applying this argument shows that $F^{\otimes P}$ satisfies Assumption \ref{ass:subcategory-free-modules}\eqref{ass:subcategory-free-modules(b)} at $N$, and so the corollary follows from Proposition \ref{prop:tensorproductcoefficients}.
\end{proof}

\subsubsection{The destabilization complex} For a coefficient system $F$ of finite degree, the range in which the map \eqref{eq:coefficient-stabilization} is an isomorphism is determined by the connectivity of (the geometric realization of) a certain semi-simplicial set, or in many cases (as those we will consider), by its underlying simplicial complex:

\begin{definition}
    Let $W_n(A,X)_\bullet$ be the semi-simplicial set whose set of $p$-simplices is defined by
$$W_n(A,X)_p:=\Hom_{U\cG}\left(X^{\oplus p+1},A\oplus X^{\oplus n}\right),$$
with the face map
$$d_i:W_n(A,X)_{p+1}\to W_n(A,X)_p$$
given by precomposition with $\id_{X^{\oplus i}}\oplus\iota_X\oplus\id_{X^{\oplus p-i}}$ (we refer the reader to \cite[Section 2]{WahlRW} for more details).

To the semi-simplicial set $W_n(A,X)_\bullet$ we also associate a simplical complex $S_n(A,X)$, whose set of vertices is given by $W_n(A,X)_0$, and where a $(p+1)$-tuple $([B_0,g_0],\ldots,[B_p,g_p])$ of vertices spans a $p$-simplex in $S_n(A,X)$ if and only if these are the vertices of a $p$-simplex in $W_n(A,X)_p$, i.e.\ if and only if there is $[B,g]\in\Hom_{U\cG}(X^{\oplus (p+1)},A\oplus X^{\oplus n})$ such that
$$[B_i,g_i]=[B,g]\circ (\iota_{X^{\oplus i}}\oplus \id_{X}\oplus\iota_{X^{\oplus(p-i)}}),$$
for each $0\le i\le p$.
\end{definition}

\noindent The connectivity of that semi-simplicial set $W_n(A,X)_\bullet$ is determined by the of the simplicial complex $S_n(A,X)$ under the following mild standardness assumptions on the category $U\cG$:

\begin{definition}\label{def:loc-stdness}
    Let $\cG$ be a symmetric monoidal groupoid and $A,X\in\cG$. We say that $U\cG$ is \textit{locally standard} at $(A,X)$ if
\begin{enumerate}[(i)]
        \item\label{LS1} $\iota_A\oplus\id_X\oplus \iota_X$ and $\iota_{A\oplus X}\oplus\id_X$ are distinct in $\Hom_{U\cG}\left(X,A\oplus X^{\oplus 2}\right)$,
        \item\label{LS2} for all $n\ge 0$, the map $(-\oplus \iota_X):\Hom_{U\cG}\left(X,A\oplus X^{\oplus n}\right)\to \Hom_{U\cG}\left(X,A\oplus X^{\oplus (n+1)}\right)$ is injective.
    \end{enumerate}
\end{definition}

\noindent The relationship between the connectedness of $W_n(A,X)_\bullet$ and $S_n(A,X)$ can then be stated as follows:

\begin{proposition}[{\cite[Proposition 2.9, Theorem 2.10]{WahlRW}}]\label{prop:connectivity-W-vs-S}
    Let $\cG$ be a symmetric monoidal groupoid with no zero divisors and such that $\Aut_{\cG}(0)=\{\id_0\}$. If $A,X\in\cG$ are such that $U\cG$ satisfies injectivity, local cancellation and local standardness at $(A,X)$ then, for any $n\ge 0$ and any $a,k\ge 1$, the simplicial complex $S_n(A,X)$ is $(n-a)/k$-connected if and only if the semi-simplicial set $W_n(A,X)_\bullet$ is $(n-a)/k$-connected.
\end{proposition}

\subsubsection{The stability theorem} We can now state the main stability theorem of \textcite{WahlRW}:

\begin{theorem}\label{thm:RW-W-stability}
Let $\cG$ be a symmetric monoidal groupoid with no zero divisors and where $\Aut_{\cG}(0)=\{\id_0\}$ and let $A,X\in\cG$ a pair where local cancellation and injectivity are satisfied. Assume that for every $n\ge 1$, there is a $k\ge 2$ and $a\ge 2$ such that the destabilization space $W_n(A,X)_\bullet$ is $\frac{n-a}{k}$-connected. If $F:\mathcal{C}_{A,X}\to\Ab$ is a coefficient system of degree $r$ at $N\in\bZ$, the map 
$$\left(-\oplus\id_X,F(\sigma_n^X)\right)_*:H_*(G_n,F_n)\to H_*(G_{n+1},F_{n+1})$$
is an isomorphism in the range $*\le\frac{n+2-a}{k}-r-1$, for all $n>N+2-a$. 
\end{theorem}

\begin{proof}
As in the proof of Proposition \ref{prop:extendcoefficientstofunctor} above, the assumptions that $\cG$ has no zero divisors and that local cancellation and injectivity hold at $(A,X)$ imply that $\cC_{A,X}$ is a pre-braided homogeneous category. For $a=2$, the statement for a coefficient system of degree $r$ is thus precisely the special case of \cite[Theorem A]{WahlRW} where the category $\cC$ is $\cC_{A,X}$.

\noindent If $a>2$, we apply the argument from \cite[Remark 3.5]{WahlRW}, but let us spell it out. We set $A':=A\oplus X^{\oplus a-2}$, $G_n':=\Aut_{\cG}\left(A'\oplus X^{\oplus n}\right)$ and $F':\cC_{A',X}\to\Ab$ by restricting $F$. Note that if $F$ is  of degree $r$ at $N\in\bZ$, then $F'$ is of degree $r$ at $N+2-a$. Furthermore, we have that
$$W_n(A',X)_\bullet=W_n(A\oplus X^{\oplus a-2},X)_\bullet$$
is $\frac{n-2}{r}$-connected, so if $F$ is of degree $r$, it follows that 
$$H_*(G_n',F_n')\to H_*(G_{n+1}',F_{n+1}')$$
is an isomorphism for $*\le \frac{n}{k}-r-1$, for all $n>N+2-a$. Since $G_n'=G_{n+a-2}$ and $F_n':=F_{n+a-2}$, this gives us the stated range.
\end{proof}

\subsection{Proof of Theorem \ref{thm:twostabilizations}}

Let us now apply this framework to our stabilization maps $\phi_n^s$ and $\nu_n^s$. We start by introducing two symmetric monoidal groupoids which play the role of $\cG$ in the framework above, for our two respective stabilizations. We use the perspective of Section \ref{sec:MCG-3mflds} and consider categories where the objects are $3$-manifolds, since this gives us destabilization complexes whose connectivity was determined by Hatcher and Wahl in \cite{HatcherWahl3manifolds}.

\subsubsection{The symmetric monoidal groupoid $\cG(m)$} Our symmetric monoidal groups are defined similarly in both cases, so let us define a family of groupoids which includes them both. 

\begin{definition}
    For $m\ge 1$, let $\cG(m)$ denote following monoidal groupoid: \begin{itemize}
        \item The objects of $\cG(m)$ are triples $(M,s,b)$, where either\begin{enumerate}
            \item $M$ is a smooth, compact, connected and oriented $3$-manifold with $s\ge m$ boundary components, numbered from $1$ to $s$, all of which are spheres, together with a parametrization $b:\bigsqcup^s S^2\overset{\cong}{\longrightarrow} \partial M$ of the boundary,
            \item or $M=\sqcup^m D^3$, $s=m$ and $b:\sqcup^m S^2\to \sqcup D^3$ is given by the inclusion of the boundary in each connected component.
        \end{enumerate}  
        We will typically leave $s$ and $b$ implicit in the notation. 
        \item A morphism $(M,s,b)\to (M',s',b')$ is an isotopy class of orientation preserving diffeomorphisms $M\to M'$ that preserve the parametrizations of the boundaries (in particular, it is not allowed to permute the boundary components), modulo Dehn twists along embedded 2-spheres in the source and target manifold.
        \item Note that the parametrization of the boundary gives us well-defined notions of \textit{equator}, as well as \textit{northern} and \textit{southern hemispheres} in each boundary component of an object $(M,s,b)$. For manifolds $M,M'\in\cG(m)$, with $s$ and $s'$ boundary components, respectively, we define a monoidal product $M\oplus M'$ by gluing the northern hemisphere and equator of the $1$st marked boundary component of $M$ to the southern hemisphere and equator of the $1st$ marked boundary component of $M'$, and for each $i\in\{0,\ldots,m-2\}$, the gluing the northern hemisphere and equator of the $(s-i)$th boundary component of $M$ to southern hemisphere of the $(s'-i)$th boundary component of $M'$. This is homotopic to gluing the manifolds along a 3-dimensional pairs of pants (i.e.\ $S^3$ with three disjointly embedded open 3-disks removed) at each of the $m$ chosen boundary components. The $s+s'-m$  parametrized boundary components in $M\oplus M'$ are ordered so that the glued ones come first, then the remaining $s-m$ ones from $M$ and lastly the $s'-m$ ones from $M'$. 
        \item The monoidal unit is $0:=(\sqcup^m D^3, m,b)$.
    \end{itemize}
\end{definition}

\begin{figure}[h]
    \centering
    \includegraphics[scale=0.25]{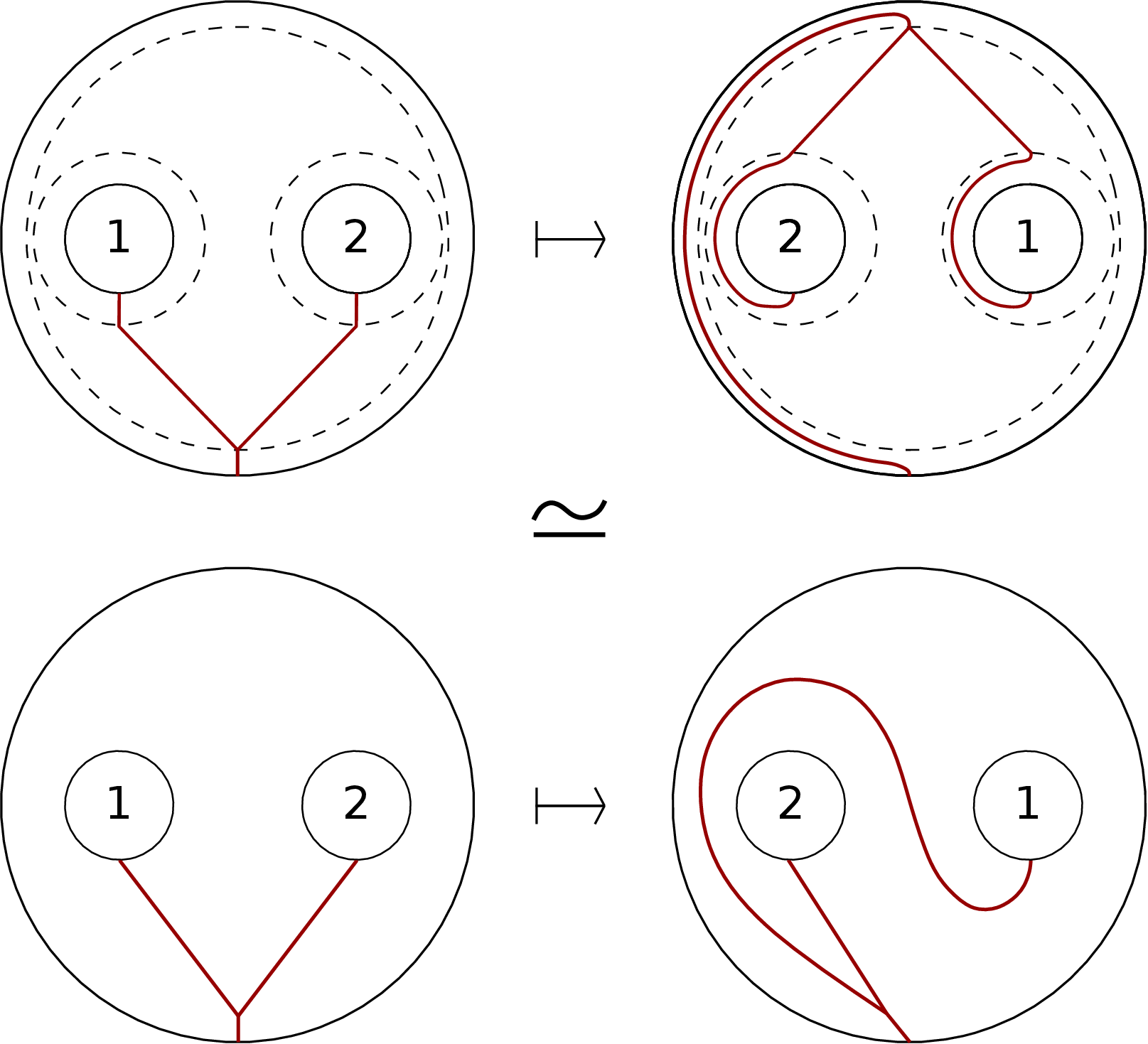}
    \caption{The 2-dimensional version of the symmetry in $\mathcal{G}(m)$. The ``outer'' half Dehn twist switches the places of the ``inner'' boundary components, while the other two half Dehn twists make sure the automorphism fixes the boundary components. An isotopic automorphism which looks a bit simpler is illustrated as in the bottom picture.}
    \label{fig:braiding}
\end{figure}

\begin{remark}
    It might not be immediately clear why it is well-defined to quotient by pre- and postcomposition by Dehn twists along embedded 2-spheres in the definition of the morphisms in $\cG(m)$. To see this, we use that if $\psi:M\to N$ is a diffeomorphism of $3$-manifolds, $S\subset M$ is a $2$-sphere, and we denote by $T_S$ the Dehn twist around $S$, then it is easy to verify that $\psi\circ T_S\circ\psi^{-1}=T_{\psi(S)}$. If we have two composable diffeomorphisms $\phi:L\to M$ and $\psi:M\to N$ and $S\subset M$ is an embedded $2$-sphere, we thus have $\psi\circ T_S\circ \phi=T_{\psi(S)}\circ\psi\circ\phi$.
\end{remark}

\begin{remark}
    It might seem more natural to define the monoidal product by gluing the $i$th boundary component of $M$ to the $i$th boundary component of $M'$ for each $i\in\{1,\ldots,m\}$, but we have chosen this somewhat strange convention in order to match with our definition of the stabilization map $\nu_n^s:A_n^s\to A_{n+1}^{s}$ (see Figure \ref{fig:nu-manifold} below). In fact, below we only consider the monoidal groupoid $\cG(m)$ for $m=1,2$, in which case we either glue along the first boundary component, or the first and last boundary component.
\end{remark}

\begin{remark}
    This monoidal product is not actually strictly monoidal, but can be ``strictified'' in the same way as in \cite[Section 3]{GalatiusKupersRW-Cells}, by instead considering manifolds whose boundary is a cube $I^3$, with two chosen faces, and gluing along these. Similarly to the examples considered in \cite[Sections 5.6-5.7]{WahlRW}, we will work with the non-strict version, since it is somewhat more practical in our proofs. 
\end{remark}

\begin{proposition}
    The monoidal category $(\cG(m),\oplus, 0)$ has no zero divisors and $\Aut_{\cG(m)}(0)=\{\id_0\}$.
\end{proposition}

\begin{proof}
    If $M\oplus M'=0$, then $M$ and $M'$ must have been disconnected in the first place, so since $0$ is the only disconnected manifold in our category, the first claim follows. If $f\in\Aut_{\cG(m)}(0)$, we have by definition $f\circ b=b$, so $f$ cannot permute the connected components. By the Alexander trick, it thus follows that $f$ must be trivial. The last part now follows by Proposition \ref{prop:underlying-groupoid}.
\end{proof}

\noindent By Proposition \ref{prop:underlying-groupoid}, in particular we get the following corollary as a consequence:

\begin{corollary}
     For any $A\in\cG(m)$, we have $\Aut_{\cG(m)}(A)=\Aut_{U\cG(m)}(A)$
\end{corollary}

\noindent We define a symmetry on $\cG(m)$ by choosing a neighborhood of each gluing, which is a 3-dimensional pair of pants, and taking the isotopy class of the automorphism of each pair of pants which is given by doing half a Dehn twist in a neighborhood of each of its boundary components. A 2-dimensional analogue of this automorphism is illustrated in Figure \ref{fig:braiding}.

\begin{figure}[h]
    \centering
    \includegraphics[scale=0.6]{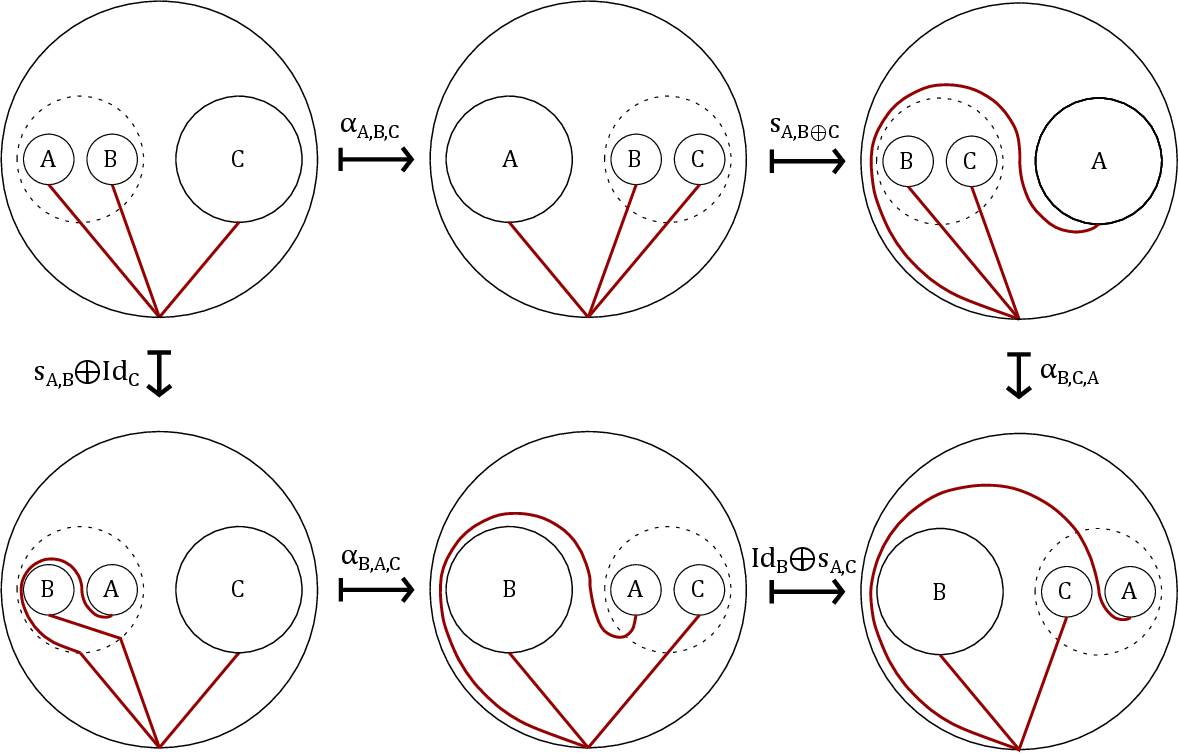}
    \caption{Illustration of the associativity coherence of the symmetry in the symmetric monoidal groupoid $\cG(m)$, where $s$ denotes the symmetry and $\alpha$ is the associator.}
    \label{fig:associativityrel}
\end{figure}

 \noindent The square of this automorphism is given by Dehn twisting a neighborhood of all three boundary components in each pair of pants. In the observation (I) of \cite[Section 2]{HatcherWahl3manifolds}, it is explained why such an automorphism of a 3-dimensional pair of pants is isotopic to the identity, meaning that our symmetry squares to the identity. Furthermore, from Figure \ref{fig:associativityrel} we see that it satisfies the braid relation, so $\cG(m)$ is a symmetric monoidal groupoid.

\subsubsection{The complex of arcs and coconnected sphere systems} For each stabilization we consider, we need to determine the connectivity of its associated destabilization complex. In order to do this, we show in each case that it is isomorphic to a special case of a complex defined by Hatcher and Wahl in \cite[Section 4.2]{HatcherWahl3manifolds} (see also \cite[Section 2]{HatcherStability}), who determined its connectivity. To define this complex, we first need to introduce the notion of a \textit{sphere system} in a $3$-manifold $M$:

\begin{definition}
    Let $M$ be an oriented 3-manifold, possibly with boundary. A $(p+1)$-sphere system in $M$ is a smooth embedding $f:\sqcup^{p+1} S^2\hookrightarrow M$, where no individual embedded sphere bounds a $3$-ball in $M$ or is isotopic to a boundary sphere and where the embedded spheres are pairwise non-isotopic. We call such a system \textit{coconnected} if $M\setminus\Image(f)$ is a connected manifold.
\end{definition}

\noindent The isotopy class of a sphere system is determined by the isotopy classes of the individual embedded spheres, by the following result of Laudenbach:

\begin{lemma}\label{lem:homotopic-sphere-systems-are-isotopic}
    Homotopic sphere systems in $M$ are isotopic.
\end{lemma}

\begin{proof}
    For a single embedded sphere, this is {\cite[Théorème III.1.3.]{Laudenbach-Sphere-Systems} and \cite[Lemma on p. 124]{Laudenbach-Sphere-Systems}} tells us that this extends to sphere systems.
\end{proof}

\noindent We will show that the simplices in our destabilization complexes correspond to isotopy classes of sphere systems, together with some additional data. With $I=[0,1]$ denoting the unit interval, let $S^2+I$ be the quotient of $S^2\sqcup I$, given by identifying the midpoint of $I$ with a base point in $S^2$. For $M$ an oriented 3-manifold with boundary, we choose points $x_0$ and $x_1$ in the boundary of $M$ (where $x_0=x_1$ is allowed). 

\begin{definition}[{see \cite[Section 4.2]{HatcherWahl3manifolds}}]
     We define $X^{\mathrm{A}}(M,x_0,x_1)$ to be the simplicial complex whose $0$-simplices are isotopy classes (relative to $x_0$ and $x_1$) of maps $f:S^2+I\to M$, with the following properties:\begin{enumerate}[(i)]
    \item $f$ is smooth on $S^2$ and $I$,
    \item the restriction of $f$ to $S^2+(0,1)$ is an embedding into $\mathring M$,
    \item $f(I)$ and $f(S^2)$ intersect transversely in $M$,
    \item $f(0)=x_0$ and $f(1)=x_1$ and
    \item the induced orientations of $f(I)$ and $f(S^2)$ combine to give the orientation on $M$.
\end{enumerate} 
     Furthermore, $(f_0,\ldots,f_p)$ forms a $p$-simplex in $X^{\mathrm{A}}(M,x_0,x_1)$ if whenever $i\neq j$, the image of $f_i$ is disjoint from that of $f_j$ (up to isotopy fixing $x_0$ and $x_1$), outside of the endpoints, and $\sqcup_i f_i:\sqcup^{p+1} S^2\to M$ defines the isotopy class of a coconnected sphere system. 
\end{definition}

\begin{remark}
    The fact that this is a well-defined simplicial complex follows from Lemma \ref{lem:homotopic-sphere-systems-are-isotopic}.
\end{remark}

\noindent Hatcher and Wahl proved the following connectivity for this simplicial complex:

\begin{proposition}{\cite[Proposition 4.5]{HatcherWahl3manifolds}}
For $M$ a compact, connected and oriented $3$-manifold, the complex $X^{\mathrm{A}}(M,x_0,x_1)$ is $\frac{n-3}{2}$-connected, where $n$ is the number of $S^1\times S^2$-summands in the prime decomposition of $M$ as a connected sum.
\end{proposition}

\begin{figure}[h]
    \centering
    \includegraphics[scale=0.23]{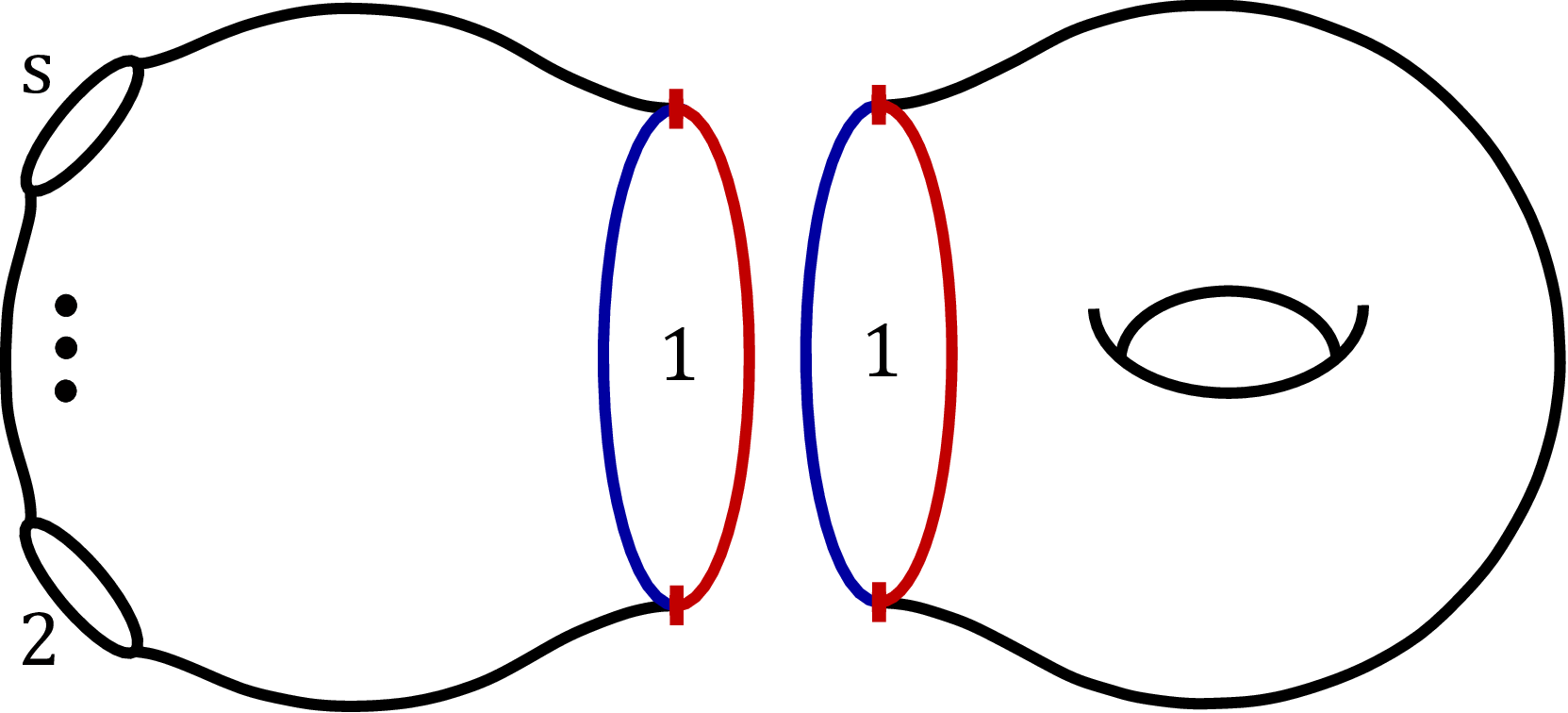}
    \caption{A 2-dimensional analogue of how to construct $A_1\oplus X_1\in U\cG(1)$, by gluing the northern hemisphere (red) of the boundary of the left manifold to the southern hemisphere (blue) of the boundary of the right manifold.}
    \label{fig:phi-manifold}
\end{figure}

\subsubsection{The stabilization $\phi_n^s$} We now consider the stabilization map $\phi_n^s$. For $s\ge 1$, we let $A_1\in\cG(1)$ be $S^3\setminus(\sqcup^s \mathring{D}^3)$, and let $X_1\in\cG(1)$ be the manifold $(S^1\times S^2)\setminus \mathring{D}^3$. Analogues in two dimensions of $A_1$ and $X_1$ and how to glue them to get $A_1\oplus X_1$ are illustrated in Figure \ref{fig:phi-manifold}. It follows that $\Aut_{U\cG(1)}\left(A_1\oplus X_1^{\oplus n}\right)\cong \Gamma_n^s$ and the stabilization map $-\oplus\mathrm{id}_{X_1}$ is the standard inclusion map $\phi_n^s:\Gamma_{n}^s\hookrightarrow \Gamma_{n+1}^s$, which is known to be injective (see for example \cite[Proposition 2.3]{HatcherWahl3manifolds}). Local cancellation holds due to the prime decomposition theorem for 3-manifolds. In addition, we have the following:

\begin{proposition}\label{prop:LS-for-UG(1)}
    The category $U\cG(1)$ satisfies local standardness at $(A_1,X_1)$.
\end{proposition}

\begin{proof}
    The proof is elementary, so we sketch the main points and leave the remaining details as an exercise to the interested reader. For brevity, let us write $X=X_1$ and $A=A_1$ throughout the proof.
    
    We start by noting that for any $n\ge 1$, the mapping sets in $U\cG$ are defined precisely so that as sets we have
    \begin{equation}\label{eq:homs-in-UG-as-sets}
        \Hom_{U\cG}(X,A\oplus X^{\oplus n})\cong A_n^s/A_{n-1}^s,
    \end{equation}
    where we consider $A_{n-1}^s$ as a subgroup of $A_n^s$ via the (injective) stabilization map $\phi_{n-1}^s:A_{n-1}^s\to A_n^s$. 
    
    To verify \eqref{LS1} of Definition \ref{def:loc-stdness}, note that by the description of the symmetric monoidal structure on $U\cG(1)$ from Proposition \ref{prop:mononidal-structure-on-UG}, we have
    \begin{align}
        \iota_A\oplus\id_{X}\oplus\iota_X&=[A\oplus X,\id_{A\oplus X^{\oplus 2}}\circ(\id_A\oplus s_{X,X}^{-1})] ,\label{LS1-first-map}\\
        \iota_{A\oplus X}\oplus\id_X&=[A\oplus X,\id_{A\oplus X^{\oplus 2}}]\label{LS1-second-map}.
    \end{align}
    In terms of the identification \eqref{eq:homs-in-UG-as-sets}, we identify \eqref{LS1-first-map} as represented by the element $(\tau,1,\ldots,1)\in A_2^s$, where $\tau\in \Aut(F_2)$ is the transposition of the standard generators. Meanwhile, \eqref{LS1-second-map} is represented by the identity element and an elementary calculation shows that these are not in the same coset in $A_2^s/A_1^s$.  

    To verify \eqref{LS2} of Definition \ref{def:loc-stdness}, we note that in terms of the identification \eqref{eq:homs-in-UG-as-sets}, the map
    $(-\oplus\iota_X):\Hom_{U\cG}(X,A\oplus X^{\oplus n})\to\Hom_{U\cG}(X,A\oplus X^{\oplus (n+1)})$
    can be described as follows: similarly to above, let $\tau\in\Aut(F_{n+1})$ denote the transposition of the standard generators $x_n$ and $x_{n+1}$ and let $c_\tau: A_{n+1}^s\to A_{n+1}^s$ be given by the diagonal action of $\tau$ (i.e.\ via conjugation on the first factor and directly via the action on $F_{n+1}$ on the remaining factors). In these terms, we identify the map $(-\oplus\iota_X)$ as the map induced on the cosets by the composite map $c_\tau\circ\phi_n^s:A_n^s\to A_{n+1}^s$. Again, an elementary direct calculation shows that the induced map of sets $A_n^s/A_{n-1}^s\to A_{n+1}^s/A_n^s$ is injective.    
\end{proof}

\noindent By Proposition \ref{prop:connectivity-W-vs-S}, it follows that for any $n\ge 0$, the connectivity of $W_n(A_1,X_1)_\bullet$ is equal to that of $S_n(A_1,X_1)$. In order to obtain homological stability with finite degree coefficients, it thus only remains to compute the connectivity of the simplicial complex $S_n(A_1,X_1)$. We prove this by comparing it to the complex of arcs and coconnected sphere systems introduced above:

\begin{proposition}\label{prop:firstcomplexiso}
    For $n\ge 1$, let $M=A_1\oplus X_1^{\oplus n}$ and $x_0=x_1$ be the north pole of the first component of $\partial M$. Then we have an isomorphism of simplicial complexes $S_n(A_1,X_1)\cong X^{\mathrm{A}}(M,x_0,x_1)$.
\end{proposition}

\begin{proof}
We recall that a 0-simplex of $S_n(A_1,X_1)$ is an equivalence class of pairs $(B,g)$, where $B\in\cG(1)$ and $g:B\oplus X_1\to A_1\oplus X_1^{\oplus n}$ is an isomorphism in $\cG(1)$, where $(B,g)\sim (B',g')$ if there is an isomorphism $h:B\to B'$ in $\cG(1)$ such that $g=g'\circ(h\oplus\id_{X_1})$. 

Let $[B,g]\in S_n(A_1,X_1)_0$ and let $\tilde{g}:B\oplus X_1\to A_1\oplus X_1^{\oplus n}$ be a boundary preserving diffeomorphism such that the isotopy class of $\tilde{g}$ is equal to $g$, modulo Dehn twists around 2-spheres. By precomposing with the canonical smooth embedding $X_1\hookrightarrow B\oplus X_1$, this gives us an embedding $\iota_g:X_1\hookrightarrow A_1\oplus X_1^{\oplus n}$.

We may assume that the parametrization of the boundary is chosen so that there is a point $y\in S^2$ such that the circle $S^1\times\{y\}\subset S^1\times S^2$ intersects the boundary of the deleted ball $\mathring D^3$ of $X_1$ only in $x_0$. Picking some $t\in [0,1]$ such that  $\{t\}\times S^2\subseteq X_1$ does not intersect the boundary, and which intersects $S^1\times\{y\}$ transversely, we thus get a map $f_0:S^2+I\to X_1$. We thus get a vertex $f$ of $X^{\mathrm{A}}(M,x_0,x_1)$ by taking the isotopy class of $\iota_g\circ f_0$. 

To show that this gives us a well-defined map $S_n(A_1,X_1)_0\to X^{\mathrm{A}}(M,x_0,x_1)_0$, we need to show that $f$ is independent of the choice of representative $(B,g)$ as well as of the diffeomorphism $\tilde{g}$ representing $g$. The second statement holds because neither embedded circles nor embedded $2$-spheres are affected by Dehn twisting around $2$-spheres, up to isotopy. If $(B,g)\sim(B',g')$, we have $g=g'\circ(h\oplus\id_{X_1})$, for some $h:B\to B'$, so $g$ and $g'$ agree when restricted to $X_1$, proving the first statement.

This thus defines a map $\Phi:S_n(A_1,X_1)_0\to X^{\mathrm{A}}(M,x_0,x_1)_0$. To see that $\Phi$ is a bijection, we show that if $f:S^2+I\hookrightarrow A_1\oplus X_1^{\oplus n}$ is a vertex of $X^A(M,x_0,x_1)$, it determines a unique equivalence class $[B,g]\in \Hom_{U\cG}\left(X_1,A_1\oplus X_1^{\oplus n}\right)$, which is mapped to the isotopy class of $f$ under the map we have defined. To see this, we first note that given such $f$, we can take a regular neighborhood $N$ of $f(S^2+I)\subset A_1\oplus X_1^{\oplus n}$, which we may assume intersects the boundary component precisely in the (closed) northern hemisphere. Such a regular neighborhood $N$ is diffeomorphic to $X_1$, and we can parametrize the boundary component of $N$ so that its southern hemisphere precisely is the northern hemisphere of the boundary component of $A_1\oplus X_1^{\oplus n}$. If we let $B:=(A_1\oplus X_1^{\oplus n})\setminus \mathring N$, whose boundary we parametrize in a similar way, we get a boundary preserving diffeomorphism $g:B\oplus X_1\to A_1\oplus X_1^{\oplus n}$, such that $\Phi([B,g])=f$. This shows surjectivity. 

To see injectivity, suppose $(B,g)$ and $(B',g')$ are representatives of vertices in $S_n(A_1,X_1)$ such that $\Phi([B,g])=\Phi([B',g'])$. By local cancellation, we may assume $B=B'$. We want to show that $g^{-1}\circ g'=h\oplus\id_{X_1}$, for some $h:B\to B$. Letting $S^1\times\{y\}\cup\{t\}\times S^2\subset X_1$ be as before and $f_0:S^2+I\to B\oplus X_1$ be the associated map, we have by assumption that $g^{-1}\circ g'\circ f_0$ and $f_0$ are isotopic maps $S^2 + I\to B\oplus X_1$ (and the image of both are in the second summand by definition). The summand $X_1$ is a regular neighborhood of $f_0(S^2+I)$, so any representative of $g'\circ g^{-1}:B\oplus X_1\to B\oplus X_1$ can be isotoped to one of the form $h\oplus h':B\oplus X_1\to B\oplus X_1$. Furthermore, $h'$ defines the class (up to isotopy and Dehn twists around 2-spheres) of a diffeomorphism of $X_1$ which fixes $\partial X_1$, as well as set-wise fixes the embedded sphere $f_0(S^2)$ and the point $f_0(1/2)$ on this embedded sphere, which is in the interior of $X_1$. By \cite[Proposition 2.1]{HatcherWahl3manifolds}, a diffeomorphism of $X_1$ which fixes the boundary pointwise and fixes a point in the interior is, up to isotopy, either a Dehn twist around an embedded $2$-sphere in $X_1$, or given by pushing the boundary sphere along a loop in $X_1$ with its endpoints on the sphere, or is non-trivial in $\Aut(\pi_1(X_1,f_0(1/2)))=\Aut(\bZ)={\pm 1}$. The condition of fixing $f_0(S^2)$ set-wise means that the diffeomorphism restricts to a diffeomorphism $X_1\setminus f_0(S^2)$, which rules out the second possibility, and since we are modding out by Dehn twists around $2$-spheres, only the third option remains. However, the non-identity element of $\Aut(\pi_1(X_1,f_0(1/2)))$ reverses the orientation of $f_0(I)$, so the fact that $g^{-1}\circ g'\circ f_0|_I$ is isotopic to $f_0|_I$ also rules out this possibility. Thus we must have $g^{-1}\circ g|_{X_1}= \id_{X_1}$ and so $[B,g]=[B,g']$.

Finally, we need to show that $\Phi$ defines a map of simplicial complexes, i.e.\ that it sends $p$-simplices to $p$-simplices. A $(p+1)$-tuple of vertices  $[B_0,g_0],\ldots,$ $[B_p,g_p]$ in $S_n(A_1,X_1)$ defines a $p$-simplex if there exists $[B,g]\in\Hom_{\cG}(X_1^{\oplus (p+1)},A_1\oplus X_1^{\oplus n})$ whose vertices are $[B_0,g_0],\ldots,[B_p,g_p]$. In other words, if we let
$$\iota_i:=\iota_{X_1^{\oplus i}}\oplus\id_{X_1}\oplus \iota_{X_1^{\oplus (p-i)}}:X_1\longrightarrow X_1^{\oplus (p+1)},$$ then $[B_i,g_i]=[B\oplus X_1^{\oplus p},g\circ (\id_{B\oplus X_1^{\oplus p}}\oplus\iota_i)]$, for $0\le i\le p$. From this it is immediate that $\Phi$ sends a $p$-simplex to a $(p+1)$-tuple $(f_0,\ldots, f_p)$ of maps $S^2+I\to A_1\oplus X_1^{\oplus n}$, whose images are disjoint except at $x_0=x_1$. With $\{t\}\times S^2\subset X_1$ as above, cutting $X_1$ along this embedded 2-sphere leaves it connected, so the sphere system of $\Phi([B_0,g_0],\ldots,[B_p,g_p])$ is coconnected. This proves that $\Phi$ is a map of simplicial complexes and since it is a bijection on the vertices, it is an isomorphism.
\end{proof}

\noindent Applying Theorem \ref{thm:RW-W-stability} gives us the following:

\begin{theorem}\label{thm:stab1}
With $(A_1,X_1)$ as above for $s\ge 1$ and $G_n=\Aut_\mathcal{G}(A_1\oplus X_1^{\oplus n})\cong \Gamma_n^s$, and $F:\mathcal{C}_{A_1,X_1}\to\Ab$ a coefficient system of degree $r$ at $N\in\bZ $, the map
$$\left(\phi_n^s,F(\sigma_n^{X_1})\right)_*:H_*(\Gamma_n^s,F_n)\to H_*(\Gamma_{n+1}^s,F_{n+1})$$
is an isomorphism in the range $2{*}\le n-2r-3$ and $n\ge N$.
\end{theorem}

\begin{remark}
    The case $s=1$ is essentially a special case of \cite[Theorem 5.31]{WahlRW}, although we have defined our category slightly differently.
\end{remark}

\noindent For brevity, let us write $M_n^s:=A_1\oplus X_1^{\oplus n}$. We have a sequence $\{F^1_n\}_{n\ge 0}$ of $G_n$-representations given by $F_n^1:=H^*_\bZ(n,s)$, and with $G_n$-equivariant stabilization maps $\sigma_n^1:F_n^1\to F_{n+1}^1$ given by $\sigma_n^1=(m_n^{s+1}\circ a_{n+1}^s)^*$, where $m_n^{s+1}$  and $a_{n+1}^s$ are defined in Section \ref{subsubsec:twistedstabilizationmaps} and we use $(-)^*$ to denote the dual map. We know that conditions \eqref{cond-i:extend-coeff-system}-\eqref{cond-iii:extend-coeff-system} of Proposition \ref{prop:extendcoefficientstofunctor} hold. Furthermore, $\sigma_n^1$ is easily verified to satisfy condition \eqref{cond-iv:extend-coeff-system} of Proposition \ref{prop:extendcoefficientstofunctor}, so it follows that it defines a coefficient system $F^1:\cC_{A_1,X_1}\to\Ab$. 

\begin{proposition}\label{prop:firstdegree1}
    The coefficient system $F^1:\cC_{A_1,X_1}\to\Ab$ is of degree 1 at 0.
\end{proposition}

\begin{proof}
    We first need to show that $\kappa_{X_1}(F^1):\cC_{A_1,X_1}\to\Ab$ is the zero functor. We recall that for $B=M_n^s$ we have $F^1(\sigma_{X_1})_B=F^1(\sigma_{X_1}^n)$, where 
    $$\sigma_{X_1}^n=(s_{X_1,A_1}\oplus \id)\circ(\iota_{X_1}\oplus\id):M_n^s\to M_{n+1}^s.$$
    From Remark \ref{remark:upperlowersusp} we know that
    $$\sigma_{X_1}^n=(s_{X_1,A_1}\oplus \id_{X_1^{\oplus n}})\circ s_{A_1\oplus X_1^{\oplus n},X_1}\circ\sigma^{X_1}_n.$$
    Since the symmetry maps are isomorphisms, it is sufficient to show that $F^1(\sigma_n^{X_1})$ is injective. The map $F^1(\sigma_n^{X_1})=\sigma_n^1$ is the dual of the standard projection $\bZ^{n+1}\oplus\bZ^{s-1}\twoheadrightarrow \bZ^{n}\oplus\bZ^{s-1}$, and is thus injective, so it follows that $\kappa_{X_1}(F^1)$ is the zero functor. Furthermore, $\delta_{X_1}(F^1)(A_1\oplus X_1^{\oplus n})\cong\bZ$, independently of $n$, and $\delta_{X_1}(F^1)(\sigma_n^{X_1})=\id_{\bZ}$, so the cokernel has degree 0.
\end{proof}

\noindent As both $F^1$ and $\delta_{X_1}(F^1)$ take values in finitely generated free abelian groups and $\delta_{X_1}^r(F^1)=0$ for $r\ge 2$, combining this with Corollary \ref{cor:tensorpowercoefficients} and Theorem \ref{thm:stab1}, we get:

\begin{corollary}\label{cor:firststab}
    For each $r\ge 0$, there is some $N(r)\in\bZ$, such that for any any $s\ge 1$, the map
    $$(\phi_n^s,(m_n^{s+1}\circ a_{n+1}^s)^*)_*:H_*(A_n^s,H^*_\bZ(n,s)^{\otimes r})\to H_*(A_{n+1}^s,H^*_{\bZ}(n+1,s)^{\otimes r})$$
    is an isomorphism in the range $2{*}\le n-2r-3$ as long as $n\ge N(r)$.
\end{corollary}

\subsubsection{The stabilization $\nu_n^s$} Let us now consider the stabilization map $\nu_n^s:\cG_n^s\to\cG_{n+1}^s$. For $s\ge 2$, we let $A_2\in U\cG(2)$ be $S^3\setminus(\sqcup^s \mathring{D}^3)$. We let $X_2\in\cG(2)$ be the manifold $S^2\times I$. Analogues in two dimensions of $A_2$ and $X_2$ and how to glue them to get $A_2\oplus X_2$ are illustrated in Figure \ref{fig:nu-manifold}. Once again $\Aut_{U\cG(2)}\left(A_2\oplus X_2^{\oplus n}\right)\cong \Gamma_n^s$, but the stabilization map $-\oplus\id_{X_2}$ is the homomorphism $\nu_n^s:\Gamma_{n}^s\to \Gamma_{n+1}^s$. 

\begin{figure}[h]
    \centering
    \includegraphics[scale=0.23]{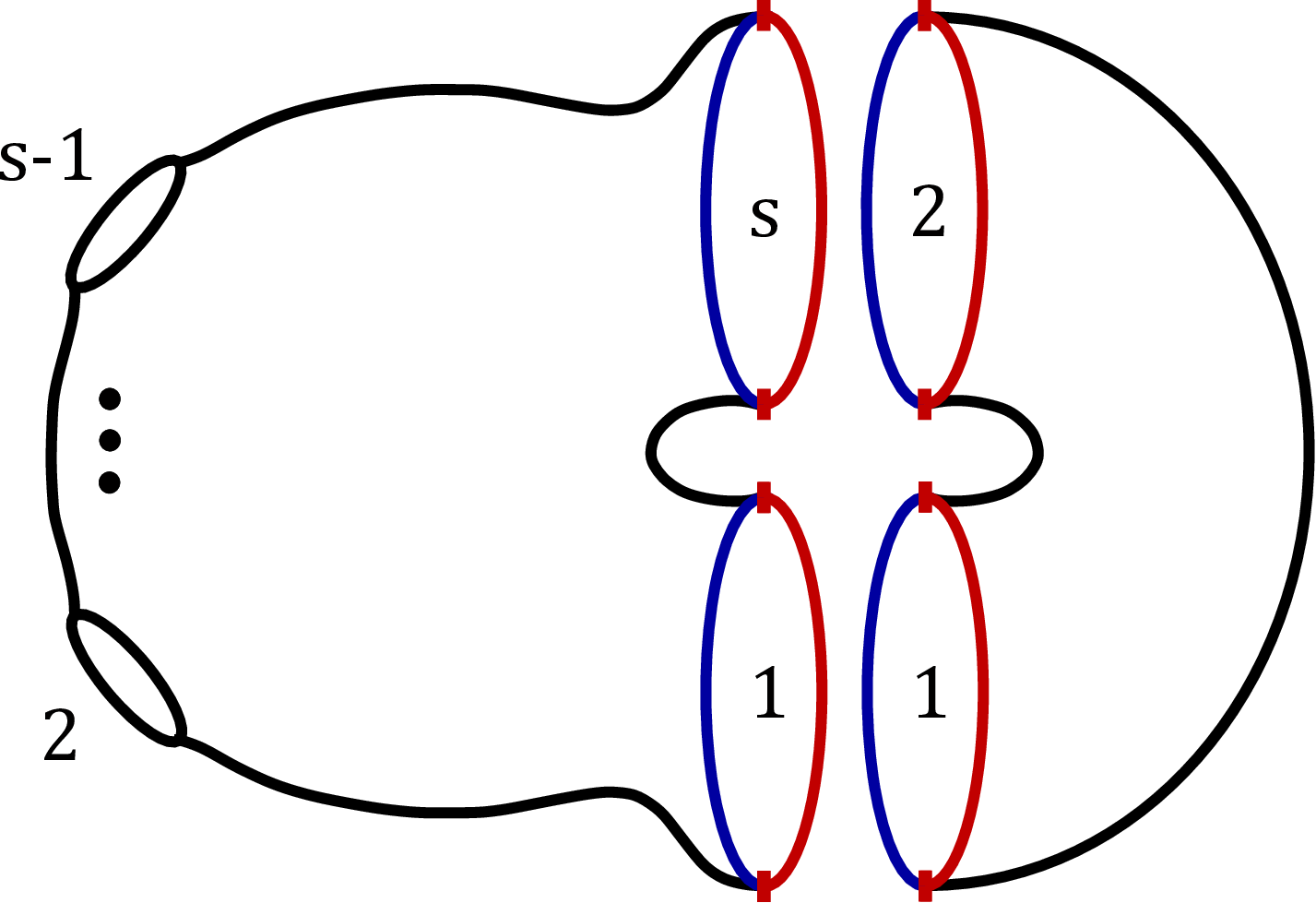}
    \caption{A 2-dimensional analogue of how to construct $A_2\oplus X_2\in U\cG(2)$.}
    \label{fig:nu-manifold}
\end{figure}

\noindent Local cancellation again follows from the prime decomposition theorem for 3-manifolds, but we need to prove injectivity, since this does not appear in the literature, to the knowledge of the author.

\begin{proposition}
The homomorphism $\nu_n^s:\Gamma_n^s\to \Gamma_{n+1}^s$ is injective for any $n\ge 1$.
\end{proposition}

\begin{proof}
Recall that $\Gamma_n^s\cong A_n^s=\Aut(F_n)\ltimes F_n^{\times (s-1)}$, where $\Aut(F_n)$ acts diagonally on $F_n^{\times (s-1)}$. We have by definition $\nu_n^s=\mu_{n+1}^{s-1}\circ\alpha_n^s$ and the homomorphism $\mu_{n+1}^{s-1}:A_{n+1}^{s-1}\to A_{n+1}^{s}$ is simply the standard inclusion $\Aut(F_{n+1})\ltimes F_{n+1}^{\times(s-2)}\hookrightarrow\Aut(F_{n+1})\ltimes F_{n+1}^{\times (s-1)}$, which is injective by definition. Thus we only need to prove injectivity of $\alpha_n^s$. Let $\{x_1,\ldots,x_{n+1}\}$ be the standard generators of $F_{n+1}$. By the descriptions of generators \eqref{AnsGen3} and \eqref{AnsGen4} of $A_n^s$ above, we have
$$\alpha_n^s(\psi,y_1,\ldots,y_{s-1})=(L_{y_{s-1},x_{n+1}}\circ\phi_{n}^1(\psi),y_1,\ldots,y_{s-2}),$$
where $\phi_n^1:\Aut(F_n)\to\Aut(F_{n+1})$ is the standard inclusion and $L_{y_{s-1},x_{n+1}}\in\Aut(F_{n+1})$ is defined by $x_{n+1}\mapsto y_{s-1}x_{n+1}$ and acting trivially on the remaining generators. 

If $(\psi,y_1,\ldots, y_{s-1})\in\ker(\alpha_n^s)$, we thus have $y_1=\cdots=y_{s-2}=1$, where we use $1$ to denote the identity element in $F_n$ (and in $F_{n+1}$), and we also have
$$x_{n+1}=(L_{y_{s-1},x_{n+1}}\circ\phi_n^1(\psi))(x_{n+1})=L_{y_{s-1},x_{n+1}}(x_{n+1})=y_{s-1}x_{n+1},$$
so $y_{s-1}=1$. Thus $L_{y_{s-1},x_{n+1}}=\mathrm{id}_{F_{n+1}}$. For $1\le k\le n$ we thus have
$$x_k=(L_{y_{s-1},x_{n+1}}\circ\phi_n^1(\psi))(x_k)=\phi_n^1(\psi)(x_k)=\psi(x_k),$$
so we also have $\psi=\mathrm{id}_{F_n}$. In summary we have 
$$(\psi,y_1,\ldots,y_{s-1})=(\mathrm{id}_{F_n},1\ldots,1),$$
proving that $\alpha_n^s$ is indeed injective.\end{proof}

\noindent Again, the connectivity of $W_n(A_2,X_2)_\bullet$ is determined by that of $S_n(A_2,X_2)$, thanks to the following proposition and Proposition \ref{prop:connectivity-W-vs-S}:

\begin{proposition}
    The category $U\cG(2)$ satisfies local standardness at $(A_2,X_2)$.
\end{proposition}

\begin{proof}
    The proof of Proposition \ref{prop:LS-for-UG(1)} essentially repeats verbatim, replacing $\phi_n^s$ by $\nu_n^s$ in the calculations, so we leave the details as an exercise to the interested reader.
\end{proof}

\noindent Note that since $A_2\oplus X_2^{\oplus n}\cong M_n^s$, its number of $S^1\times S^2$-summands is $n$. We will thus again determine the connectivity of $S_n(A_2,X_2)$ by comparing it to the complex of arcs and coconnected sphere systems:

\begin{proposition}\label{prop:simpcomplexes}
For $n\ge 1$, let $M=A_2\oplus X_2^{\oplus n}$ and $x_0$ and $x_1$ be the north poles of the first and last components of $\partial M$, respectively. Then we have an isomorphism of simplicial complexes $S_n(A_2,X_2)\cong X^{\mathrm{A}}(M,x_0,x_1)$. 
\end{proposition}

\begin{proof}
We can assume that the boundary components of $X_2=S^2\times I$ are parametrized so that if $y$ is the north pole of its first boundary component, then $\{y\}\times I$ is a line between the north poles of the two boundary components. Defining an embedding $f_0:S^2+I\hookrightarrow X_2$ by
\begin{align*}
    f_0(x)=\begin{cases}
        (y,x) &\text{if }y\in I,\\
        (x,1/2)&\text{if }x\in S^2,
    \end{cases}
\end{align*}
the argument of the proof of Proposition \ref{prop:firstcomplexiso} goes through almost word for word, using that a regular neighborhood of an embedding of $S^2+ I$ is diffeomorphic to $X_2$. Let us only remark on the points that differ.

When showing that the map between the simplicial complexes is injective, we obtain the isotopy class of a diffeomorphism of $S^2\times I$, which fixes the boundary pointwise, fixes the point $y\times (1/2)$ on the sphere $S^2\times(1/2)$ and fixes the rest of this sphere set-wise. Once again we can use \cite[Proposition 2.1]{HatcherWahl3manifolds}, but since $S^2\times I$ is simply connected, in this case such a diffeomorphism that fixes the boundary pointwise is necessarily a Dehn twist around a $2$-sphere.

In the final step, we note that $X_2\setminus(S^2\times\{1/2\})$ is disconnected, but since $B\oplus X_2^{\oplus (p+1)}$ is obtained by gluing each boundary component of the copies of $X_2$ to two different boundary components of $B$, which is connected, the corresponding sphere system is still coconnected. 
\end{proof}

\noindent From the connectedness of $S_n(A_2, X_2)$, we get our second stability result:

\begin{theorem}\label{thm:stab2}
For $s\ge 2$ and $(A_2, X_2)$ as above, $G_n=\Aut_{\cG(2)}\left(A_2\oplus  X_2^{\oplus n}\right)\cong \Gamma_n^s$ and $F:\mathcal{C}_{A_2, X_2}\to\Ab$ a coefficient system of degree $r$ at $N\in\bZ$, the map
$$\left(\nu_n^s,F(\sigma_n^{X_2})\right)_*:H_*(\Gamma_n^s,F_n)\to H_*(\Gamma_{n+1}^s,F_{n+1})$$
is an isomorphism for $2{*}\le n-2r-3$ and $n\ge N$.
\end{theorem}

\noindent Again, we have $M_n^s\cong A_2\oplus X_2^{\oplus n}$. We define a sequence $\{F_n^2\}_{n\ge 0}$ of $G_n$-representations by $F_n^2:= H^*_{\bZ}(n,s)$ and $G_n$-equivariant maps $\sigma_n^2:F_n^2\to F_{n+1}^2$ by $\sigma_n^2:=(a_{n+1}^{s-1}\circ m_{n+1}^s)^*$. Once again it is clear that the hypotheses of Proposition \ref{prop:extendcoefficientstofunctor} are satisfied, so we get a coefficient system $F^2:\cC_{A_2,X_2}\to\Ab$.

\begin{proposition}
    The coefficient system $F^2:C_{A_2,X_2}\to\Ab$ is of degree 1 at 0.
\end{proposition}

\begin{proof}
The proof goes through much in the same way as that of Proposition \ref{prop:firstdegree1}, but let us spell out the details for completeness. To prove that $\kappa_{X_2}(F^2)$ is zero, it is once again enough to show that 
$$\sigma_n^{X_2}:H^*_\bZ(n,s)\to H^*_\bZ(n+1,s)$$
is injective for each $n\ge 0$. This map is the dual of the map 
$$(a_{n+1}^{s-1}\circ m_{n+1}^s):H_\bZ(n+1,s)\cong\bZ^{n+1}\oplus\bZ^{s-1}\to \bZ^{n}\oplus\bZ^{s-1}\cong H_\bZ(n,s)$$ given by
$$(x_1,\ldots,x_{n+1},y_1,\ldots,y_{s-1})\mapsto (x_1,\ldots,x_n,y_1,\ldots,y_{s-2},x_{n+1}),$$
which is a surjection, so $\sigma_n^{X_2}$ is indeed injective. Furthermore, $\delta_{X_2}(F^2)(A_2\oplus X_2^{\oplus n})\cong \bZ$, for all $n\ge 0$, and $\delta_{X_2}(F^2)(\sigma_n^{X_2})=\id_\bZ$, so $\delta_{X_2}(F^2)$ has degree 0.\end{proof}

\noindent Once again it is clear that both $F^2$ and $\delta_{X_2}(F^2)$ take values in finitely generated free abelian groups and that $\delta_{X_2}^r(F^2)=0$ for $r\ge 2$, so combining this with Corollary \ref{cor:tensorpowercoefficients} and Theorem \ref{thm:stab1}, we get:

\begin{corollary}\label{cor:secondstab}
    For each $r\ge 0$, there is some $N(r)\in\bZ$ such that for $s\ge 2$, the map
    $$(\nu_n^s,(a_{n+1}^{s-1}\circ m_{n+1}^s)^*)_*:H_*(A_n^s;H_\bZ^*(n,s)^{\otimes r})\to H_*(A_{n+1}^s, H_\bZ^*(n+1,s)^{\otimes r})$$
    is an isomorphism in the range $2{*}\le n-2r-3$ whenever $n\ge N(r)$.
\end{corollary}

\noindent Combining Corollaries \ref{cor:firststab} and \ref{cor:secondstab}, we get a dual version of Theorem \ref{thm:twostabilizations}. The cohomological stability theorem is a direct consequence of the following universal coefficient theorem for group cohomology, a proof of which appears in \cite[Lemma 3.5]{RW18}:

\begin{lemma}
    Let $G$ be a group and $M$ a left $\bZ[G]$-module. Writing $M^*:=\Hom_{\bZ}(M,\bZ)$ for the dual (right) $\bZ[G]$-module. There is a natural short exact sequence
    $$0\to \mathrm{Ext}_R^1(H_{i-1}(G,M),\bZ)\to H^i(G,M^*)\to \Hom_{\bZ}(H_i(G,M),\bZ)\to 0.$$
\end{lemma}

\begin{remark}
    Combining Corollaries  \ref{cor:firststab} and \ref{cor:secondstab} also shows that $\alpha_n^s$ induces an isomorphism in homology with the coefficients $H_{\bZ}^{*}(n,s)^{\otimes r}$, in a stable range. Furthermore, we can note that if $\pi_n^s:A_n^{s+1}\to A_n^{s}$ is the projection, we have $\pi_n^s\circ\mu_n^s= \id_{A_n^s}$, so $\pi_n^s$ also induces an isomorphism in homology with these coefficients, in a stable range.
\end{remark}

\vspace{-0.2em}
\printbibliography
\vspace{-0.8em}
\Addresses

\end{document}